\documentclass[10pt,twoside]{article}

\usepackage{amsmath}
\usepackage{amsfonts}
\usepackage{amssymb}
\usepackage{amscd}
\usepackage{amsthm}
\usepackage{amsbsy}
\usepackage{graphicx}
\usepackage{bm}
\usepackage{xypic}
\def\cal{\mathcal}
\def\Bbb{\mathbb}

\def\G{\Gamma}
\def\r{\rangle}
\def\l{\langle}
\def\t{\times}

\newtheorem{prop}{Proposition}[subsection]
\newtheorem{thm}{Theorem}[subsection]
\newtheorem{exm}{Example}[subsection]

\newtheorem{lemma}{Lemma}[subsection]
\newtheorem{cor}{Corollary}[subsection]

\newtheorem*{fjic}{(Fibered) Isomorphism conjecture}
\newtheorem{defn}{Definition}[subsection]

\newtheorem{rem}{Remark}[subsection]
\def\@oddhead{\hfill \shorttitle \hfill \thepage}
\def\@evenhead{\thepage \hfill \shortauthor \hfill}
\def\@oddfoot{}
\def\@evenfoot{}

\date{January 27, 2014}
\title{\ \\[0.4cm] \ \\ \bf  The isomorphism conjecture for groups with 
generalized free product structure}
\author{Sayed K. Roushon\footnote{School of Mathematics, Tata Institute of 
Fundamental Research, Homi Bhabha Road, Mumbai-400005, India.}\hspace{2mm}}
\begin{document}
%-------------------

\maketitle
%------------------------

\thispagestyle{empty}

%--------------------------------------

\begin{abstract}
\vskip 3mm\footnotesize{

\vskip 4.5mm
\noindent
In this article we study the $K$- and $L$-theory 
of groups acting on trees. We consider   
the problem in the context of the fibered isomorphism conjecture of 
Farrell and Jones. 
We show that in the class of residually finite 
groups it is enough to prove the conjecture 
for finitely presented groups with one end. Also, we deduce that the conjecture is 
true for the fundamental groups of graphs of finite groups and of trees of virtually 
cyclic groups. To motivate the reader we include a 
survey on some classical works on this subject.

\vspace*{2mm} \noindent{\bf 2000 Mathematics Subject Classification:
} Primary: 19J05, 19J10, 19J25, 19D35; Secondary: 57N37.

\vspace*{2mm} \noindent{\bf Keywords and Phrases: }} reduced projective class group, 
Whitehead group, surgery groups, 
fibered isomorphism conjecture, group action on tree.

\end{abstract}
\vspace*{-11.7cm}
%%%%%%%%%%%%%%%%%%%%%%%%%%%%%%%%%%%%%
\vspace*{2mm} \noindent\hspace*{82mm}
\begin{picture}(41,10)(0,0)\thicklines\setlength{\unitlength}{1mm}
\put(0,2){\line(1,0){41}} \put(0,16){\line(1,0){41}}%
\put(0,12.){\sl \copyright\hspace{1mm}Higher Education Press}
\put(0,7.8){\sl \hspace*{4.8mm}and International Press}%
\put(0,3.6){{\sl \hspace*{4.8mm}Beijing-Boston} }
\end{picture}

\vspace*{-17.6mm}\noindent{{\sl Handbook of\\
Group Actions}\\ALM\,31, pp.\,1--?} \vskip8mm
%%%%%%%%%%%%%%%%%%
\vspace*{10.6cm}
\section{Introduction} We study the computation of 
$K$- and $L$-theory, in terms of generalized homology theories,   
of groups acting on trees from the computation of the respective theories for the  
vertex and the edge stabilizers. We also give a survey on some classical works 
on this subject.

To motivate the reader we begin with some of the fundamental questions in 
Topology, which are answered by the $K$- and $L$-theory of groups.

$\bullet$ When is a space homotopy equivalent to a 
finite $CW$-complex? 

$\bullet$ When is a $CW$-complex homotopy equivalent to a 
manifold? 

$\bullet$ Are two homotopy equivalent closed manifolds homeomorphic? When the manifolds are  
closed and aspherical a positive answer is expected and it is also known as the {\it Borel 
Conjecture}. 

$\bullet$ When can we place a boundary to a noncompact manifold?

$\bullet$ When is a map $f:M\to {\Bbb S}^1$ from a manifold $M$  
homotopic to a fiber bundle projection?
 
These questions were well studied over the last 
several decades. Lower $K$-theory (that is, the Whitehead group $Wh(G)$, reduced 
projective class group $\tilde K_0({\Bbb Z}[G])$, negative $K$-groups $K_{-i}({\Bbb Z}[G])$, $i\geq 1$)   
and the surgery $L$-groups $L_n({\Bbb Z}[G])$ of 
the fundamental group $G$ contain the answers to the above and many more questions.
The subject is enormous and, therefore, we describe some of these obstruction 
groups and the contexts.

\subsection{Lower $K$-groups and surgery groups}

One can easily construct examples to answer the above questions in the negative. Therefore,  
we need several necessary conditions. 

A connected space $X$ is 
called {\it finitely dominated} if there is a connected finite complex $Y$ and 
maps $f:X\to Y$ and $g:Y\to X$ so that $g\circ f$ is homotopic to the identity 
map of $X$. Note that this is half of saying that $X$ is homotopy equivalent to a finite 
complex. It is also known that if $X$ is finitely dominated then it is 
homotopy equivalent to a countable complex, which can be checked using 
some facts from \cite{M}. Therefore, to investigate whether 
a space is homotopy equivalent to a finite complex we need to make the assumption that 
$X$ is finitely dominated. The {\it Wall finiteness obstruction} (\cite{Wall1}, \cite{Wall2}) is produced from a finitely 
dominated connected space $X$, which is an element $\omega (X)\in \tilde K_0({\Bbb Z}[\pi_1(X)])$
so that $X$ is homotopy equivalent to a finite complex if and only if $\omega(X)=0$ in 
$\tilde K_0({\Bbb Z}[\pi_1(X)])$. Furthermore, given an element $\omega$ in $\tilde K_0({\Bbb Z}[\pi_1(X)])$ 
there is a finitely dominated space whose obstruction is $\omega$. Let us recall that 
the {\it reduced projective class group} $\tilde K_0({\Bbb Z}[G])$ is by definition the 
free abelian group generated by the isomorphism classes $[P]$ of finitely generated projective 
${\Bbb Z}[G]$-modules $P$ modulo the following two relations. 
$$[P]+[Q]-[P\oplus Q], [F]$$ where $F$ is a finitely generated free 
${\Bbb Z}[G]$-module. See \cite{Ro} for some more basic information on this group.

The fourth problem also can be solved using the reduced projective class group. 
So let $M$ be a noncompact connected smooth orientable manifold. We would like 
to know whether there exist a compact manifold $N$ so that $N-\partial N$ is 
diffeomorphic to $M$.
Obviously the first obstruction will be if the fundamental group of $M$ is
not finitely presented and hence $M$ should have finitely many
ends to start with. Secondly the manifold should have finitely
generated homology. Thirdly, near the ends $M$ should 
behave like a compact
manifold product with an interval. As a consequence a necessary condition
is that outside a big enough compact subset the manifold should have the 
homology of a compact manifold of one less dimension. For
simplicity assume that the manifold has only one end. Another
necessity is that if $Y_i$ is an increasing sequence of compact subsets
of $N$, then the fundamental groups of the complements of $Y_i$ should stabilize
to a finitely presented group $\pi_1(U)=G$ at a certain stage, where $U$ is 
the complement of $Y_i$ for some large $i$. Such an end is also called a {\it tame} 
end. Under these necessary
conditions one shows that $C_*(\tilde U)$ is chain equivalent to a finite
chain complex $P_*(U)$ of finitely generated projective 
$R={\Bbb Z}[\pi_1(U)]$-modules. 

In \cite{Si} it was shown that if the Euler characteristic 
$\chi(P_*(U))=\Sigma_i (-1)^i[P_i(U)]=0$ in $\tilde K_0(R)$ and if dim$U\neq
3,4,5$ then $M$ is
the interior of a compact manifold with boundary.
For a detailed account on this work see \cite{We}.

If we want a finite complex to be homotopy equivalent to a closed orientable  
manifold there are several 
more necessary conditions needed. The first 
condition is that the homology of the complex must satisfy 
Poincar\'{e} duality. Such a complex is called a {\it Poincar\'{e}
complex}.
Let $X$ be a connected finite complex and for some $n$, $H_n(X, {\Bbb Z})\simeq {\Bbb Z}$. 
Let $[X]\in H_n(X, {\Bbb Z})$ be a generator such that the cap product 
with $[X]$ gives an isomorphism $H^q(X, {\Bbb Z})\to H_{n-q}(X, {\Bbb Z})$ for all $q$. Then $X$ is called 
a {\it Poincar\'{e} complex} with {\it orientation} $[X]$ and dimension $n$.

The second condition needed is the existence of a bundle over the complex, 
which has properties similar to the normal bundle of a manifold embedded 
in some Euclidean space. We describe this below.

Let $M$ be a closed connected oriented smooth manifold of dimension $n$. By the Whitney 
Embedding Theorem we can embed this manifold in 
${\Bbb R}^{n+k}\subset {\Bbb S}^{n+k}$ for $k \geq n$. Let $\nu_M$ be the 
normal bundle of $M$ in ${\Bbb S}^{n+k}$. Let $\tau_M$ be the 
tangent bundle of $M$. Then $\tau_M\oplus \nu_M$ is the product bundle as 
$M\subset {\Bbb S}^{n+k}-\{ \infty \} ={\Bbb R}^{n+k}$. Let $N$ be the subset 
of $\nu_M$ consisting of vectors of length $< \epsilon$, with respect to some Riemannian 
metric, for some $\epsilon > 0$. Then 
$N$ is an open neighborhood of $M$ in ${\Bbb S}^{n+k}$ and in fact diffeomorphic 
to the total space $E(\nu_M)$ of $\nu_M$. The one point compactification 
of $T(\nu_M)$ is called the {\it Thom space} of $\nu_M$. On the other hand the 
one point compactification $N^*$ of $N$ is homeomorphic to ${\Bbb S}^{n+k}/{\Bbb S}^{n+k}-N$. Hence, 
we get a map $\alpha:{\Bbb S}^{n+k}\to N^*\simeq T(\nu_M)$. One can check that 
$\alpha$ induces an isomorphism $H_{n+k}({\Bbb S}^{n+k}, {\Bbb Z})\to 
H_{n+k}(T(\nu_M), {\Bbb Z})$ and sends the canonical generator of 
$H_{n+k}({\Bbb S}^{n+k}, {\Bbb Z})$ to the generator of 
$H_{n+k}(T(\nu_M), {\Bbb Z})$, which comes from the orientation class 
$[M]$ via the Thom isomorphism. In this sense this map is of degree $1$. 

Therefore, for a Poincar\'{e} complex $X$ to be homotopy equivalent to a closed 
oriented smooth manifold it is necessary that there should be a bundle $\xi$ on $X$ and a 
degree $1$ map $\alpha:{\Bbb S}^{n+k}\to T(\xi)$. Once this information is given 
we can apply Thom Transversality Theorem to homotope $\alpha$ to $\beta$ so that 
$\beta^{-1}(X)$ ($:=K$) is a (oriented) submanifold of ${\Bbb S}^{n+k}$. Furthermore,  
the map $\beta$ restricted to the normal bundle $\nu_K$ of $K$ gives a linear bundle map onto 
$\xi$ and $\beta |_K$ is of degree $1$. A data of this type as in the following commutative 
diagram is called a {\it normal map} and is denoted by $(f,b)$.

\centerline{
\xymatrix{\nu_M\ar[r]^f\ar[d] & \xi\ar[d]\\
M\ar[r]^b & X}}

Where, $M$ is a closed connected smooth oriented manifold, $X$ is a Poincar\'{e} complex, 
$\nu_M$ is the normal bundle in some embedding of $M$ in ${\Bbb S}^{n+k}$, and $b$ is 
a degree 1 map. As we saw this map $b$ is obtained from the Transversality 
Theorem but is nowhere close to be a homotopy equivalence. In the simply connected and 
odd high dimension case in fact $b$ can be 
homotoped to a homotopy equivalence (\cite{Br}). So the next general step is to 
apply Surgery theory to $b$ to get another normal map, which is normally cobordant to the 
previous one, and to try to get closer to a homotopy equivalence. To achieve this we  
need that $b$ induces isomorphisms on the homotopy groups level. This can be 
done up to dimension $<[\frac{n}{2}]$ and, therefore, by Poincar\'{e} duality the 
main problem lies in dimension $[\frac{n}{2}]$. The complication in 
this middle dimension gives rise to Wall's {\it Surgery obstruction groups} 
$L^h_*({\Bbb Z}[\pi_1(X)])$. That is, given a normal map $(f,b)$, there is an obstruction 
$\sigma(f,b)$, which lies in the group $L^h_*({\Bbb Z}[\pi_1(X)])$, whose vanishing will ensure that the normal map 
can be normally cobordant to another normal map $(f', b')$, where $b'$ is a homotopy equivalence. See 
\cite{Wall3}.

\begin{rem}{\rm 
These surgery 
groups depend only on the fundamental group and on its orientation character 
$\omega: \pi_1(X)\to {\Bbb Z}_2$. Here as we are dealing with the oriented case 
this homomorphism is trivial. In the general situation we need to add the $\omega$ in the 
notation of the surgery groups. But in this article we avoid it for simplicity.}\end{rem} 

\begin{rem}\label{h=s}{\rm The upper script $h$ to the notation of the surgery groups is 
due to `homotopy equivalence'. There are problems, when one asks for `simple homotopy 
equivalence' in the normal map. Then a different surgery problem appears and gives 
rise to the surgery groups $L_n^s({\Bbb Z}[-])$. There are many other decorated surgery 
groups for different surgery problems, like $L_n^{\langle -\infty\rangle }({\Bbb Z}[-])$ 
mentioned before. But all of them coincide once we have the 
lower $K$-theory vanishing result of the group. This is checked using Rothenberg's exact 
sequence. For example if the Whitehead group of a group $G$ vanishes, then,  
$L_n^h({\Bbb Z}[G])=L_n^s({\Bbb Z}[G])$ for all $n$. 

One further remark is that all surgery groups are $4$-periodic. 
That is, for example, $L_n^h({\Bbb Z}[-])=L_{n+4}^h({\Bbb Z}[-])$ for all $n$.}\end{rem}

Once we have established that a complex is homotopy equivalent to a manifold, the 
next question is about the uniqueness of such a manifold. In surgery theory one gets 
two such manifolds $M_1$ and $M_2$ (when they exist) as $h$-cobordant, that is, 
there is a manifold $W$ with two boundary components homeomorphic to $M_1$ and 
$M_2$ and the inclusions $M_i\subset W$ are homotopy equivalences. Given such an  
$h$-cobordism $W$ there is an obstruction $\tau(W, M_1)$, which lies in a quotient of the 
$K_1$ of the integral group ring of the fundamental group of $M_1$ defined below. 

Let $\G$ be a group and ${\Bbb Z}[\G]$ be its integral group ring. $K_1({\Bbb Z}[\G])$ is by 
definition $\frac{GL({\Bbb Z}\G)}{[GL({\Bbb Z}\G), GL({\Bbb Z}\G)]}$ where 
$GL({\Bbb Z}\G)=\lim_{n\to\infty}GL_n({\Bbb Z}\G)$. There is a map 
from $\G$ to $K_1({\Bbb Z}[\G])$ sending an element $g\in \G$ to the 
$1\times 1$ matrix $<g>$. Define the {\it Whitehead group} $Wh(\G)$
of $\G$ as the quotient $\frac{K_1({\Bbb Z}[\G])}{\{ \pm <g>\ |\ g\in \G \}}$.  

Now, given an $h$-cobordism $W$ between two connected manifolds $M_1$ and $M_2$ 
with dim$W\geq 6$ there is an element $\tau(W, M_1)\in Wh(\pi_1(M_1))$ 
with the property: $\tau(W, M_1)=0$ implies that $W$ is homeomorphic to $M_1\times I$  
where $I=[0,1]$. This is called the {\it $s$-cobordism Theorem}. See \cite{Ker} and 
\cite{KS}. One consequence of this theorem is the Poincar\'{e} conjecture in 
high dimensions. Furthermore, given an element $\tau\in Wh(\pi_1(M_1))$ 
there is an $h$-cobordism $W$ over $M_1$ realizing $\tau$. 

There is another 
interpretation of the Whitehead group, which says that given a 
homotopy equivalence $f:K\to L$ between two connected finite complexes $K$ and $L$ there 
is an element $\tau(f)\in Wh(\pi_1(K))$ whose vanishing ensures that the 
map $f$ is homotopic to a {\it simple homotopy equivalence}. See \cite{Co}.

There are similar interpretation of the reduced projective class groups and 
negative $K$-groups in terms of some `special' kind of $h$-cobordism called 
{\it bounded h-cobordism}. See \cite{Pe} for some more on this matter. 

For the last problem in the list we started with  
one needs $Wh(-)$, $\tilde K_0({\Bbb Z}[-])$ and one 
more obstruction group called the Farrell Nil group. See \cite{F}. We describe 
the steps briefly. Let $f:M\to {\Bbb S}^1$ be a map where $M$ is a closed connected  
smooth manifold. We need to find whether $f$ can be homotoped to a fiber 
bundle projection. At first, obviously we need $f_*:\pi_1(M)\to \pi_1({\Bbb S}^1)$ to 
be surjective. Secondly, the kernel of $f_*$ must be finitely presented as it 
should be the fundamental group of a fiber of a possible fibration. This can be 
achieved by looking at the infinite cyclic covering $\tilde M$ of $M$ corresponding to the kernel of 
$f_*$. Hence, $\tilde M$ must be finitely dominated and, then, we 
use the Wall finiteness obstruction, which lies in $\tilde K_0({\Bbb Z}[\pi_1(\tilde M)])$. 
Next, a Nil group is defined in \cite{F}, which contains the obstruction to 
the existence of a framed submanifold $N$ of $M$ representing $f$ so that we get  
an $h$-cobordism $(\overline{M-N}, N_1, N_2)$. Here $N_1$ and $N_2$ are the two 
boundary components of $\overline{M-N}$. Finally, we need to make this $h$-cobordism 
a product to achieve a fiber bundle projection. There comes the dimension 
restriction and a Whitehead torsion type obstruction 
for the application of the $s$-cobordism theorem. 

\subsection{Methods of computations}

For computation of the lower $K$-groups and surgery groups there are generally 
two methods. 

\medskip
\noindent
{\bf Method A.} Using the following two classical assembly maps,

$$A_K:H_*(BG, {\Bbb K})\to K_*({\Bbb Z}[G])$$ and 
$$A_L:H_*(BG, {\Bbb L})\to L_*({\Bbb Z}[G]).$$

It is still an open problem that $A_K$ and $A_L$ are isomorphisms  
for torsion free groups, which also imply the Borel conjecture. 
Recall that the Novikov conjecture asks for the  
rational injectivity of $A_K$ and $A_L$.
In the case where $G$ has torsion, $A_K$ and $A_L$ need not be surjective or injective. 

There are now two issues. First, we need to study the maps $A_K$ and $A_L$,   
and then compute the domain homology theories using the Atiyah-Hirzebruch 
spectral sequence. 

In this method, the breakthrough came in the 
works of F. T. Farrell, W. C. Hsiang and L. E. Jones. See, for example, \cite{FH81}, \cite{FH83}, 
\cite{FJ86}, \cite{FJ89} and \cite{FJ93}. They introduced geometric 
methods to study the assembly maps, and were profoundly successful in 
proving the Borel and the Novikov conjectures for flat, hyperbolic, and more 
generally for non-positively curved Riemannian manifolds. There are 
results based on these works for several other classes of groups. See, for 
example, \cite{afr0}, \cite{BFJP}, \cite{FL}, \cite{FR}, \cite{R-2}, 
\cite{R-1} and \cite{R1} to \cite{R7}. Their methods 
were extended to prove the 
Borel conjecture in the hyperbolic and $CAT(0)$ groups cases (\cite{BL1}). The 
Frobenius induction technique was also fruitfully used in \cite{FH78}, \cite{FJ88}, \cite{S} and 
in \cite{BFL}.

\medskip
\noindent
{\bf Method B.} Groups built from 
the following two classical constructions inductively (also known as {\it 
generalized free product} (\cite{W})):

$\bullet$ Amalgamated free product $G=G_1*_HG_2$;

$\bullet$ $HNN$-extension $G=K*_L$. 

For example, surface groups and 
fundamental groups of Haken $3$-manifolds (\cite{W0})  
can be built from the trivial group using the above two constructions.
The combination of the two constructions inductively (also known as 
{\it generalized free product structure} on the group (\cite{W}))
is also equivalent to the situation when the group acts on a tree. 
Consequently, given a group $G$ acting on a tree one would like to compute 
the $K$ and $L$-theory of $G$ in terms of the computation for the 
vertex and edge stabilizers of the action.

In the next section we follow the second method to get the computation of the 
$K$ and the $L$-theory of new classes of groups from known cases.

There is a generalization of the assembly map we discussed above 
to have a suitable set up for groups with torsion. Farrell and Jones 
replaced the domain of the assembly map by some other homology theory 
and conjectured ({\it (fibered) isomorphism conjecture} \cite{FJ}) the new 
assembly maps to be isomorphisms for all 
groups. This new homology theory incorporates the information   
of the obstruction groups of the virtually cyclic subgroups of the 
group. In other words, the conjecture says that the $K$ and $L$-theory of any 
group are concentrated at the $K$ and $L$-theories of the virtually 
cyclic subgroups of the group. When restricted to torsion free groups these new assembly maps  
give the classical ones above. This conjecture has profound 
applications in Geometry, Topology and Algebra and imply 
many of the standard conjectures in high dimensional topology, for example the  
Borel conjecture, the Novikov conjecture, vanishing of the lower $K$-theory for 
torsion free groups, etc. 

We will discuss this aspect of the subject in Section 3 and will  
study the isomorphism conjecture in the context of groups acting on 
trees. One basic result we prove is that to prove the fibered isomorphism 
conjecture for the class of residually 
finite groups it is enough to prove the fibered isomorphism conjecture for 
finitely presented residually finite groups with one end. We also investigate 
cases where the vertex stabilizers belong to some well-known classes of groups; for 
example virtually cyclic groups, finitely generated abelian groups, polycyclic 
groups and nilpotent groups.

To end this Introduction we mention that the lower $K$-theory information is 
already embedded in the surgery theory. So it helps in deduction of many topological 
question, which needs surgery theory if we have an ad hoc information of the 
lower $K$-theory. Sometimes to avoid this problem one defines surgery 
theory (denoted by $L_*^{\langle -\infty\rangle }$) from where the embedded lower $K$-theory information 
is removed. For example, in the isomorphism conjecture in $L$-theory the 
assembly map is expected to be an isomorphism for the $L_*^{\langle -\infty\rangle }$-theory only.

%\newpage
\section{Lower $K$-theory and surgery $L$-theory of groups acting on trees}

In this section we survey some classical works on the $K$- and $L$-theory of groups acting 
on trees and also mention some recent developments.

\subsection{Lower $K$-theory of groups acting on trees}

We recall the work from \cite{W} on the $K$-theory 
of groups with generalized free product structure. 
Under certain (regularity) conditions a 
Mayer-Vietoris type exact sequence is produced in \cite{W} 
for $K$-theory and as a consequence the author 
computed lower $K$-theory for several classes of groups including 
knot groups, and more generally fundamental groups of submanifolds of the $3$-sphere and 
compact Haken $3$-manifolds. In the general situation the existence of a  
Mayer-Vietoris type exact sequence is obstructed by certain Nil 
groups. These Nil groups vanish if some algebraic assumptions are made 
on the group ring of the group. 

Recall that by the Bass-Serre theory, given a group acting on a tree 
(without inversion) there is an associated graph of groups whose 
fundamental group is isomorphic to the group. See \cite{DD}.

Let us 
recall that a {\it graph of groups} consists of a graph 
$\cal G$ (that is a one-dimensional $CW$-complex) and 
to each vertex $v$ or edge $e$ of $\cal G$ there is associated a group 
${\cal G}_v$ (called the {\it vertex group} of the vertex $v$) 
or ${\cal G}_e$ (called 
the {\it edge group} of the edge $e$) respectively with the assumption that 
for each edge $e$ and for its two end vertices $v$ and $w$ (say)
there are injective group homomorphisms ${\cal G}_e\to {\cal G}_v$ and 
${\cal G}_e\to {\cal G}_w$. (If $v=w$ then we also demand two 
injective homomorphisms as above). The {\it fundamental group} 
$\pi_1({\cal G})$ of a finite graph $\cal G$ is defined inductively so that 
in the simple cases of graphs of groups where the graph has 
two vertices and one edge or one vertex and one edge the 
fundamental group is the amalgamated free product or the 
$HNN$-extension respectively. For an infinite graph it is then 
defined taking limit on finite subgraphs. 
See \cite{DD} for some more on this subject.

Throughout the article we make the following conventions: 

We use the same notation for a graph of groups and
its underlying graph. $V_{\cal G}$ and $E_{\cal G}$ are  
respectively the set of all vertices and edges of a graph $\cal G$.

\begin{thm} (\cite{W})\label{gg} Let $\cal G$ be a graph of groups with $\pi_1({\cal G})\simeq G$ 
and $\{ G_v\}_{v\in V_{\cal G}}$ and 
$\{ G_e\}_{e\in E_{\cal G}}$ are the vertex and edge groups of $\cal G$ respectively. Then there is 
an exact sequence:
$$\oplus_{e\in E_{\cal G}}Wh(G_e)\to \oplus_{v\in V_{\cal G}}Wh(G_v)\to Wh(G)\to$$$$\to Nil\oplus 
(\oplus_{e\in E_{\cal G}}\tilde K_0({\Bbb Z}[G_e]))\to 
\oplus_{v\in V_{\cal G}}\tilde K_0({\Bbb Z}[G_v])$$

\noindent
where $Nil$ is the Nil group.\end{thm}
 
 The lower $K$-theory is also computed as the low degree homotopy groups of a certain space
called the {\it Whitehead space} defined in \cite{W}. The homotopy groups, denoted by $Wh_i(-)$, 
of the Whitehead space then become obstructions for the $K$-theory to form a generalized  
homology theory. In other words we have the following.

$$\cdots\to H_i(BG, {\Bbb K})\to K_i(G)\to Wh_i(G)\to H_{i-1}(BG, {\Bbb K})\to\cdots .$$

Now, there is a trick which helps in computation using the above theorem. This  
says that if $G$ acts on a tree and $A$ is another group then $G\times A$ also acts on the 
same tree with stabilizers product of the stabilizers of $G$ with $A$. In the particular case  
where $A$ is free abelian, it is known that $\tilde K_0({\Bbb Z}[G\times A])$ is a subgroup of 
$Wh(G\times A\times {\Bbb Z})$. Therefore, from an a priori information that the Nil 
group $Nil$ vanishes one gets many interesting vanishing results using the Five Lemma  
argument. 

Let us now talk about under what conditions the Nil group vanishes. 

\begin{prop} (\cite{W}) Let $G$ be as above. Then $Nil=0$ if for any edge group $G_e$ 
the group ring ${\Bbb Z}[G_e]$ is regular coherent.\end{prop}

Here it is interesting to note that the vanishing of the Nil group depends 
only on the edge groups. It does not depend on the vertex groups or on 
the way the edge groups are embedded in the group $G$.

Let us recall that a ring $R$ is called {\it coherent} if its finitely presented 
modules form an abelian category and it is called {\it regular coherent} if, in addition, 
each finitely presented $R$-module has a finite-dimensional projective 
resolution. If a group $G$ acts on a tree there is a way to check 
if ${\Bbb Z}[G]$ is regular coherent.

\begin{rem}{\rm Here we remark that if ${\Bbb Z}[G]$ is regular coherent then $G$ is torsion free. 
See [\cite{BL}, Lemma 4.3].}\end{rem}

\begin{prop} (\cite{W}) Let $G$ be as above. Then ${\Bbb Z}[G]$ is regular coherent if ${\Bbb Z}[G_v]$ is regular 
coherent for each vertex $v$ and ${\Bbb Z}[G_e]$ is regular noetherian for each 
edge $e$.\end{prop}

Therefore, if a group $G$ can be constructed inductively by amalgamated free product and 
$HNN$-extension from the trivial group then ${\Bbb Z}[G]$ is regular coherent.
For example if $G$ is one of the following groups: a free abelian group, a free group, the fundamental group of 
a $2$-manifold and the fundamental group of a Haken $3$-manifold (\cite{W0}), then ${\Bbb Z}[G]$ is regular coherent.

The exact sequence in the theorem is obtained from the low degree part of the long 
homotopy exact sequence of a certain fibration. At first, the problem is reformulated in the 
un-reduced $K$-theory case. We describe this below briefly.

First, the concept of generalized free product of groups is transferred to the category 
of rings and free product with amalgamation and generalized Laurent extensions of rings 
are defined as the ring theoretic setup of the problem. Next, passing to Quillen's  
$K$-theory spaces from rings, the problem of the relation between various 
constituent spaces is analyzed. For example in the case of amalgamated free 
product of rings one has rings $A,B$ and $C$, where $C$ is embedded as a ring in $A$ and $B$ in a 
{\it pure} (\cite{W}) way, the resulting ring $R$ is the colimit of the associated diagram. 
Then the existence of a Nil space $\tilde K{\cal {NIL}}(C, A',B')$ is proved and 
it is shown that the loop space $\Omega K(R)$ of the $K$-theory space of $R$ is the  
direct product, up to homotopy, of 
$\tilde K{\cal {NIL}}(C, A',B')$ and the homotopy fiber of a map 
$K(C)\to K(A)\times K(B)$. And a similar situation occurs in the generalized Laurent 
extension case. The homotopy groups of the Nil space computes the Nil groups. 
The vanishing of the Nil groups is guaranteed by the following theorem.

\begin{thm}(\cite{W})
If the ring $C$ is regular coherent then 
the corresponding Nil space is contractible.\end{thm} 

The following Mayer-Vietoris 
type exact sequence is then deduced from the long homotopy exact sequence of 
a certain fibration.

$$\cdots\to K_i(C)\to K_i(A)\oplus K_i(B)\to K_i(R)\to K_{i-1}(C)\to\cdots\to K_0(R).$$

When $C$ is not regular coherent the exotic Nil terms appear with $K_*(C)$. Finally, 
to complete the proof of the theorem one needs the Mayer-Vietoris sequence of 
generalized homology theory and the two long exact sequences above.

Now, we state a corollary.

\begin{cor} (\cite{W}) \label{white} There is a class of groups $\cal {CL}$ containing the following 
groups such that for any $G\in {\cal {CL}}$ the Whitehead space of $G$ is contractible and 
hence their lower $K$-theory vanish.  

$\bullet$ free groups and free abelian groups.

$\bullet$ poly-infinite cyclic groups.

$\bullet$ torsion free one-relator groups.

$\bullet$ fundamental groups of 2-manifolds different from the projective plane.

$\bullet$ fundamental groups of compact orientable 3-manifolds any irreducible summand of which 
either has non-empty boundary, or is simply connected, or is Haken.

$\bullet$ fundamental groups of submanifolds of the 3-sphere.

$\bullet$ subgroups of the groups listed above.\end{cor}

\subsection{Surgery $L$-theory of groups acting on trees}
The problem in the surgery $L$-theory case was studied in 
\cite{Ca1}, \cite{Ca2}, \cite{Ca3}, \cite{Ca4} and \cite{Sh}. 
Many variations of the problem were studied 
by mathematicians before. Here we recall the above works, which give a 
strong generalization of previous results.
As in the $K$-theory case here also similar $\mathrm{UNil}$ obstruction groups were defined, which 
together with some lower $K$-theory non-triviality obstruct Mayer-Vietoris type 
exact sequence for surgery groups. Again under 
some group theoretic ({\it square root closed}) 
condition it was shown that the $\mathrm{UNil}$ groups vanish. Here it is to be noted that 
in contrast to the situation of the vanishing of the Nil groups  
where the sufficient condition was that only the amalgamated group be regular coherent, for the 
vanishing of the $\mathrm{UNil}$ group the square root closed nature of the embedding of the 
amalgamated group in the whole group is relevant.

Let $M$ and $N$ be closed manifolds, $K$ a codimension $1$  
submanifold of $N$ and $f:M\to N$ a homotopy equivalence. 

\medskip
\noindent
{\bf Splitting problem.} When can $f$ be homotoped to a map 
$g:M\to N$ so that $g^{-1}(K)$ is a submanifold of $M$ and 
$g|_{g^{-1}(K)}$ and $g|_{M-g^{-1}(K)}$ are homotopy equivalences? 

If such a $g$ exists then we say $f$ is {\it splittable}. 

\medskip

The answer to this problem is related to the computation of the 
surgery obstruction groups of a generalized free product in terms of 
its constituent groups. 

Let $M$ be a manifold and $N$ be a codimension $1$ submanifold of $M$. Recall that 
$N$ is called {\it two sided} in $M$ if $N$ has a neighborhood $V$ in $M$ such that 
$V-N$ has two components or equivalently the normal bundle of $N$ in $M$ 
is trivial. If both $M$ and $N$ are orientable then $N$ is always 
two sided in $M$. A subgroup $H$ of a group $G$ is called {\it square root closed} if 
for any $g\in G$, $g^{2}\in H$ implies $g\in H$.

We would also need the existence of the exact sequence involving Whitehead groups and 
reduced projective class groups of the constituent groups of a generalized 
free product as in Theorem \ref{gg}.

\begin{thm} (\cite{Ca1}) \label{cappell1} Let $M$ and $N$ be two closed connected 
manifolds of dimension 
$n\geq 5$ and $f:M\to N$ a homotopy equivalence. Let $K$ be a codimension 
$1$ submanifold of $N$. For simplicity of the presentation assume that $N-K$ has 
two components $N_1$ and $N_2$. Assume the following.

$\bullet$ $K$ is two sided in $N$,

$\bullet$ $\pi_1(K)\to \pi_1(N)$ is injective,

$\bullet$ $\pi_1(K)$ is square root closed in $\pi_1(N)$.

Then $f$ is splittable if and only if the Whitehead torsion $\tau (f)\in Wh(\pi_1(N))$ 
of $f$ is in the image of the map $$Wh(\pi_1(N_1))\oplus Wh(\pi_1(N_2))\to Wh(\pi_1(N)).$$
\end{thm}

Now, recall that there is a map defined in \cite{W} (Theorem \ref{gg}) 
$$\phi:Wh(\pi_1(N))\to \tilde K_0({\Bbb Z}[\pi_1(K)])\oplus Nil\to\tilde K_0({\Bbb Z}[\pi_1(K)]).$$ 

\begin{thm} (\cite{Ca1})\label{cappell2} Under the same hypothesis as in the above theorem if 
$\phi(\tau (f))=0$ then there exists an $h$-cobordism $W$ between $M$ and 
another manifold $M'$ with a map $F:W\to N$ extending $f$ so that $F|_{M'}$ is 
splittable.\end{thm}

Next, we relate the above results with the Mayer-Vietoris exact sequence for surgery 
groups. The reduced surgery groups $\tilde L_n^h({\Bbb Z}[G])$ are defined as the kernel 
of the homomorphism $L_n^h({\Bbb Z}[G])\to L_n^h({\Bbb Z}[\langle 1\rangle] )$ induced 
from the trivial map $G\to \langle 1 \rangle$. 

First, we deduce a consequence of Theorem \ref{cappell1}.

\begin{thm} (\cite{Ca1})\label{cappell3} Assume $M$, $N$ and $f$ are as above and 
$N$ is the connected sum of two manifolds $N_1$ and $N_2$. 
Assume the dimension of $M$ is $\geq 5$ and 
the fundamental groups of $N_1$ and $N_2$ have no elements of order $2$. Then 
$M$ is also a connected sum of two manifolds $M_1$ and $M_2$ and $f|_{M_i}:M_i\to N_i$ is a 
homotopy equivalence for $i=1,2$.\end{thm} 

The above theorem is not true if the fundamental group has an element of order $2$. 
In \cite{Ca5} it was shown that there is a smooth $4k+1$ dimensional manifold $M$ ($k\geq 1$) and a 
simple homotopy equivalence $M\to {\Bbb {RP}}^{4k+1}\#{\Bbb {RP}}^{4k+1}$ so that 
$M$ is not a nontrivial connected sum in any of the categories: smooth, piecewise linear or 
smooth.

Now, we come to the more general existence of a Mayer-Vietoris exact 
sequence for surgery groups. Below we state it for the situation of 
amalgamated free product. A similar statement is also there for the $HNN$-extension case. 

\begin{thm} (\cite{Ca2}, \cite{Ca3})\label{cappell4} Let $H$ be a square root closed subgroup of two groups $G_1$ and $G_2$. 
Let $G=G_1*_HG_2$ and $A$ be the kernel of the map 
$\phi:Wh(G)\to \tilde K_0({\Bbb Z}[H])$ defined above. 

$(1)$ Then there is a 
long exact sequence of surgery groups.
$$\cdots L^h_n({\Bbb Z}[H])\to L_n^h({\Bbb Z}[G_1])\oplus L_n^h({\Bbb Z}[G_2])\to 
L^A_n({\Bbb Z}[G])\to$$$$\to^{\partial} L^h_{n-1}({\Bbb Z}[H])\to\cdots.$$
Here $L^A_n({\Bbb Z}[-])$ consists of $\sigma(f,b)$ of those normal maps $(f,b)$ for 
which the Whitehead torsions of $b$ lie in the 
subgroup $A$.

$(2)$ Let $B$ be the kernel of the map $$\tilde K_0({\Bbb Z}[H])\to \tilde K_0({\Bbb Z}[G_1])\oplus \tilde K_0({\Bbb Z}[G_2]).$$ 
Therefore, $B$ is isomorphic to $Wh(G)/A$. Then there is an exact sequence 

$$\cdots\to H^{n+1}({\Bbb Z}_2; B)\to L_n^A({\Bbb Z}[G])\to L_n^{Wh(G)}({\Bbb Z}[G])\to$$$$\to H^n({\Bbb Z}_2; B)\to\cdots$$

\noindent
where the ${\Bbb Z}_2$-action comes from the orientation character $\omega:G\to {\Bbb Z}_2$ of the group $G$, 
which we suppressed in the notation of the surgery groups.\end{thm}

Here note that $L_n^{Wh(G)}({\Bbb Z}[-])=L_n^h({\Bbb Z}[-])$.

If we do not assume the square root closed condition then a new term appears 
called the $\mathrm{UNil}$-groups whose vanishing is needed for the existence of the 
above exact sequence. More explicitly, assume that $A=Wh(G)$, (which implies 
$L_n^A({\Bbb Z}[G])=L_n^h({\Bbb Z}[G])$) then 
there is a subgroup $\mathrm{UNil}_n^h({\Bbb Z}[H]; {\Bbb Z}[\hat{G_1}], {\Bbb Z}[\hat{G_2}])$ of 
$L_n^h({\Bbb Z}[G])$ so that the above Mayer-Vietoris exact sequence exists 
replacing $L_n^h({\Bbb Z}[G])$ by the quotient  
$$L_n^h({\Bbb Z}[G])/\mathrm{UNil}_n^h({\Bbb Z}[H]; {\Bbb Z}[\hat{G_1}], {\Bbb Z}[\hat{G_2}]).$$ 
Here ${\Bbb Z}[\hat{G_i}]$ is the $H$-subbimodule of ${\Bbb Z}[G_i]$ additively generated by 
$G_i-H$, for $i=1,2$.
The groups $\mathrm{UNil}_n^h({\Bbb Z}[H]; {\Bbb Z}[\hat{G_1}], {\Bbb Z}[\hat{G_2}])$ are $2$-primary and 
vanish if $G$ is square root closed. See \cite{Ca3}. The group $\mathrm{UNil}_n^h({\Bbb Z}[\langle 1\rangle]; 
{\Bbb Z}[\hat{{\Bbb Z}_2}], {\Bbb Z}[\hat{{\Bbb Z}_2]})$ for the infinite dihedral group 
${\Bbb Z}_2*{\Bbb Z}_2$ is related to the question of 
homotopy invariance of connected sum we discussed before in the context of 
projective spaces. The calculation of these $\mathrm{UNil}$ groups was only done recently in 
\cite{CD} and \cite{BR}. In \cite{Fa} it was shown that the $\mathrm{UNil}$ groups have 
exponent at the most $4$.
 
\begin{rem}{\rm We get from the above 
exact sequences that the Mayer-Vietoris type exact sequence for 
surgery groups always exists modulo $2$-torsions. Therefore, if we change the ring from ${\Bbb Z}$ to 
${\Bbb Z}[\frac{1}{2}]$ then the Mayer-Vietoris exact sequence always exists. Hence, 
using the Five Lemma, the Mayer Vietoris exact sequence above with coefficients in 
${\Bbb Z}[\frac{1}{2}]$ and the Mayer-Vietoris exact sequence for generalized homology theory 
we get that if a group $G$ acts on a tree so that for the vertex and edge stabilizers 
the assembly map in surgery theory (see Introduction) is an isomorphism then so it is 
for $G$. This was also stated in [\cite{BL}, Theorem 0.13]}\end{rem}

We know that there is a large class of (torsion free) groups for which the 
lower $K$-theory vanishes. See Section 3. Therefore, the square root closed condition is 
the only obstruction in any of these situations. We used this square root closed 
condition together with the result of \cite{W} for 
Haken $3$-manifold groups to deduce a vanishing theorem for structure sets (which is 
equivalent to saying that the assembly map is an isomorphism) for a 
large class of $3$-manifolds in \cite{R-1} and \cite{R-2}. Recently, we could deduce this 
vanishing theorem for all aspherical $3$-manifolds in \cite{R0} due to some stronger 
developments in this area.

The above theorem gives important calculations of surgery groups. 
We state the cases of free groups and surface groups. 

If $G_1$ and $G_2$ are two groups 
with no elements of order $2$ (for example if they are torsion free) then it follows from Theorem \ref{cappell4} that 
$$\tilde L_n^h({\Bbb Z}[G_1*G_2])=\tilde L_n^h({\Bbb Z}[G_1])\oplus\tilde L_n^h({\Bbb Z}[G_2])$$
and $$\tilde L^s_n({\Bbb Z}[G_1*G_2])=\tilde L^s_n({\Bbb Z}[G_1])\oplus\tilde L^s_n({\Bbb Z}[G_2]).$$ This gives the 
following calculation. 

\begin{cor} (\cite{Ca2}) Let $F_m$ be the free group on $m$ generators. Then 
$$L_n^h({\Bbb Z}[F_m])=\begin{cases}{\Bbb Z}&\text{if}\ n=4k,\\
{\Bbb Z}^m&\text{if}\ n=4k+1,\\ 
{\Bbb Z}_2&\text{if}\ n=4k+2,\\ 
{\Bbb Z}^m_2&\text{if}\ n=4k+3.\end{cases}$$
\end{cor}

Note that since the Whitehead group of the free groups vanishes (Corollary \ref{white}) the same 
calculation holds for the groups $L_n^s({\Bbb Z}[F_m])$. See Remark \ref{h=s}.

Since for the 
trivial group we already know the calculation of the surgery groups 
(restrict to $m=0$ in the free group case), let $S$ be a 
non-simply connected closed surface.

\begin{cor} (\cite{Ca2}) Let $G=\pi_1(S)$. If $S$ is orientable of genus $g\geq 1$ then 

$$L_n^h({\Bbb Z}[G])=\begin{cases}{\Bbb Z}\oplus {\Bbb Z}_2 &\text{if}\ n=4k,\\ 
{\Bbb Z}^{2g}&\text{if}\ n=4k+1,\\
{\Bbb Z}\oplus {\Bbb Z}_2&\text{if}\ n=4k+2,\\ 
{\Bbb Z}_2^{2g}&\text{if}\ n=4k+3.\end{cases}$$

If $G={\Bbb Z}_2$ then the respective calculation is the following  
$$L_n^h({\Bbb Z}[G])=\begin{cases}{\Bbb Z}_2&\text{if}\ n=4k,\\
0&\text{if}\ n=4k+1,\\ 
{\Bbb Z}_2&\text{if}\ n=4k+2,\\ 
0&\text{if}\ n=4k+3.\end{cases}$$

Finally, if $S$ is the connected sum of the Klein bottle and the orientable surface of genus  $g$ then the 
corresponding calculation takes the following form:
$$L_n^h({\Bbb Z}[G])=\begin{cases}{\Bbb Z}_2\oplus {\Bbb Z}_2&\text{if}\ n=4k,\\
{\Bbb Z}^{2g+1}&\text{if}\ n=4k+1,\\
{\Bbb Z}\oplus {\Bbb Z}_2&\text{if}\ n=4k+2,\\ 
{\Bbb Z}_2^{2g+2}&\text{if}\ n=4k+3.\end{cases}$$
\end{cor} 

Again since the Whitehead group of surface groups vanishes (Corollary \ref{white}) we have $L_n^h({\Bbb Z}[G])=L_n^s({\Bbb Z}[G])$ 
where $G$ is a surface group as above.

\begin{rem}{\rm 
In this connection, Cappell conjectured that for any knot group $G=\pi_1({\Bbb S}^3-k)$, where $k$ is a knot 
in the $3$-sphere ${\Bbb S}^3$, the abelianization 
homomorphism $G\to {\Bbb Z}$ induces an isomorphism on the surgery groups. That is,  
$L_n({\Bbb Z}[G])\to L_n({\Bbb Z}[{\Bbb Z}])$ is an isomorphism for all $n$. 
This was known as {\it Cappell's conjecture} and he 
proved it in the case where the commutator subgroup of the knot group is finitely 
generated. We proved this conjecture for all knot groups in \cite{afr}. A similar 
computation was give in \cite{R-3} for link groups. Recently,  
an explicit computation of the surgery groups of the classical Artin pure braid groups 
(denoted by $PB_m$) was given in \cite{R7}. 

\begin{thm} (\cite{R7}) If $PB_m$ is the classical Artin pure braid group on $m$ strands, then 
we have the following.

$$
L^h_n(PB_m) =\begin{cases}{\Bbb Z} & \text{if}\ n=4k,\\
{\Bbb Z}^{\frac{m(m+1)}{2}} & \text{if}\ n=4k+1,\\
{\Bbb Z}_2 & \text{if}\ n=4k+2,\\
{\Bbb Z}_2^{\frac{m(m+1)}{2}} & \text{if}\ n=4k+3.
\end{cases}$$
\end{thm}

The vanishing of the lower $K$-theory of $PB_m$ was proved in 
\cite{afr0}. Therefore, the same calculation holds for the surgery groups for 
simple homotopy equivalence.}\end{rem}

Theorem \ref{cappell4} is proved by a more general (relative) version of Theorem \ref{cappell1}.
The crucial point is to describe the boundary map $\partial$. The other 
maps are defined by the inclusion maps with proper sign. 
First, note that by Theorem \ref{gg} the image of $Wh(G_1)\oplus Wh(G_2)\to 
Wh(G)$ is the same as the kernel $A$ of $Wh(G)\to \tilde K_0({\Bbb Z}[H])$. Now, given 
$\alpha\in L_n^A({\Bbb Z}[G])$ one represents this by a normal map $(f, b)$, 
which satisfies the statement of the splitting theorem and one gets $f$ to be 
transverse to the splitting submanifold (whose fundamental group is $H$) 
and then one restricts the normal map to the inverse of this splitting 
submanifold to get a normal map giving an element in $L^h_{n-1}({\Bbb Z}[H])$. The exactness 
of the sequence is again proved by some geometric methods. For details we 
refer the reader to \cite{Ca4} and \cite{Sh}.

%\newpage
\section{The isomorphism conjecture for groups acting on trees}

In this section we study the fibered isomorphism conjecture 
of Farrell and Jones for groups acting on trees. Originally, the
conjecture was stated for pseudo-isotopy theory, algebraic $K$-theory
and for $L$-theory (\cite{FJ}). One can deduce a lower algebraic $K$-theory version 
of the conjecture (that is in dimension $\leq 1$) from the pseudo-isotopy version of the 
conjecture.
Here we prove some results for the pseudo-isotopy theory, $L$-theory and 
in the lower $K$-theory. The methods of proofs of our results hold for
all equivariant homology theories under certain conditions. We make these
conditions explicit here. 

It is well known that the pseudo-isotopy 
version of the conjecture yields the following:

$\bullet$ Computation of the Whitehead group and the lower $K$-groups of 
the associated groups.

$\bullet$ It implies the conjecture in 
lower $K$-theory (that is in dimension $\leq 1$). In particular, it proves 
that the Whitehead group, the reduced projective class group and the 
lower $K$-groups in dimension $\leq -2$ 
vanish for torsion free groups.

$\bullet$ Computation of the algebraic $K$-groups of the groups in all dimensions 
after tensoring with the rationals in terms of the ordinary rational homology groups. See 
\cite{FJ}.

$$K_n({\Bbb Z}[G])\otimes_{\Bbb Z}{\Bbb Q})\simeq H_n(BG, {\Bbb Q})\oplus\bigoplus_{k=1}^{\infty}H_{n-4k-1)}(BG, {\Bbb Q}).$$

$\bullet$ Together with the conjecture in $L$-theory this also gives the following computations  
of the homotopy groups of the space of homeomorphisms (diffeomorphisms) for 
closed aspherical manifolds $M$ of dimension $m\geq 11$ and for $1\leq i\leq \frac{m-7}{3}$. See 
\cite{FH}.

$$
\pi_i(Top(M))\otimes_{\Bbb Z}{\Bbb Q}=\begin{cases} \text{center}(\pi_1(M))\otimes_{\Bbb Z}{\Bbb Q} & \text{for}\ i=1\\ 
0 & \text{for}\ i\geq 2.\end{cases}$$ 

Also under the same hypothesis one has the following.

$$\pi_i(Diff(M))\otimes_{\Bbb Z}{\Bbb Q} =\begin{cases} \text{center}(\pi_1(M))\otimes_{\Bbb Z}{\Bbb Q}& \text{for}\ i=1\\
\bigoplus_{j=1}^{\infty}H_{(i+1)-4j}(M;{\Bbb Q})& \text{for}\ i\geq 2\ \text{and}\ m\ \text{odd}\\
0& \text{for}\ i\geq 2\ \text{and}\ m\ \text{even}.\end{cases}$$

The main problem we are concerned with is the following.

\noindent
{\bf Problem.} {\it Assume the fibered isomorphism  conjecture for the
stabilizers of the action of a group on a tree. Show that the group also
satisfies the conjecture.}

In [\cite{R1}, Reduction Theorem] we solved the problem when the 
edge stabilizers are trivial.
In this article we show it when the edge stabilizers are finite, and, in
addition, either the vertex stabilizers are residually finite or 
the group surjects onto another group with certain properties. 
We also solve the problem, under a certain condition, assuming that 
the vertex stabilizers are polycyclic. This condition is both 
necessary and sufficient for the fundamental group to be subgroup 
separable when the vertex stabilizers are finitely generated 
nilpotent. Next, we prove some results when the stabilizers are  
abelian groups. Graphs of groups of the last type were studied before 
in \cite{S} where the 
$L$-theory was computed when the stabilizers are free abelian
groups. 

A positive answer to the problem implies that the fibered 
isomorphism conjecture is true for one-relator groups and for solvable 
groups. See Subsection 3.8 for details.

\subsection{Statements of theorems}

Now, we state our main results. Before that we recall some definitions.
Given two groups $G$ and $H$ the {\it wreath product} 
$G\wr H$ is by definition the semidirect product $G^H\rtimes H$ where the 
action of $H$ on $G^H$ is the regular action and $G^H$ is the 
direct sum of copies of $G$ indexed by the elements of $H$.

In the statements of the results we say that the FIC$^P$ 
(FICwF$^P$) is true for a group 
$G$ if the Farrell-Jones fibered isomorphism  conjecture for pseudo-isotopy 
theory is true for the group $G$ (for $G\wr H$ for any finite group $H$).

The importance of the version FICwF$^P$ ({\it fibered isomorphism 
conjecture wreath product with finite groups}) was first realized in the proof of the 
conjecture for braid groups in \cite{FR}. Later it was defined in 
\cite{R1} to facilitate proving the conjecture 
for a $3$-manifold group from the knowledge of the conjecture for some 
finite sheeted covering. The main advantage in this  
version of the conjecture is that it passes to finite index 
overgroups (Proposition \ref{product}). 

\begin{thm} \label{residually} Let $\cal G$ be a graph of groups such that
the edge groups are finite. Then $\pi_1({\cal G})$ satisfies the FICwF$^P$  
if one of the following conditions is satisfied.

(1). The vertex groups are residually finite and satisfy the FICwF$^P$. 

(2). There is a homomorphism                  
$f:\pi_1({\cal G})\to Q$ onto another group $Q$ such that the restriction
of $f$ to any vertex group has finite kernel and $Q$ 
satisfies the FICwF$^P$. 
\end{thm}

This theorem is a special case of Proposition \ref{residually-prop} $(1)$ 
respectively Proposition \ref{theorem} $(b)$. See Subsection 3.7. 

We state now the following immediate corollary.

\begin{cor}\label{onend} Assume that the FICwF$^P$ is true for all 
one-ended, finitely presented residually finite groups. Then 
the FICwF$^P$ is true for all residually finite groups.\end{cor} 

Let $G$ be a residually finite group. Then by 
[\cite{FL}, Theorem 7.1] it is 
enough to consider finitely presented residually finite group. Since 
finitely presented groups are accessible (\cite{DY}), by the above Corollary 
we need to show that the 
conjecture is true for one-ended finitely presented residually finite groups 
to prove the conjecture for all residually finite groups. 

\begin{rem} {\rm There is a large class of residually
finite groups for which the FICwF$^P$  
is true. For example, any group which contains a
group from the following examples as a subgroup of finite index is a
residually finite group satisfying the FICwF$^P$.

1. Polycyclic groups (\cite{FJ}).

2. Artin full braid groups (\cite{FR}, \cite{R3}).

3. Compact $3$-manifold groups (\cite{R1}, \cite{R2}).

4. Compact surface groups (\cite{FJ}).

5. Fundamental groups of hyperbolic Riemannian manifolds (\cite{FJ}).

6. Crystallographic groups (\cite{FJ}).}\end{rem}

\begin{defn}\label{def1.1}
{\rm A graph of groups $\cal G$ is said
to satisfy the {\it intersection property} if for each 
connected subgraph of groups
${\cal G}'$ of $\cal G$, $\cap_{e\in E_{{\cal G}'}}{\cal G}'_e$
contains a subgroup which is normal in
 $\pi_1({\cal G}')$ and is of finite index in some edge group. We say  
that $\cal G$ is of {\it finite-type} if the graph is finite and 
all the vertex groups are finite.}\end{defn}

\begin{thm} \label{ip} Let $\cal G$ be a graph of groups. 
Let $\cal D$ be a collection of finitely generated
groups satisfying the following.

\begin{itemize}
\item Any element $C\in {\cal D}$ has the following properties. 
Quotients and subgroups of $C$ belong to ${\cal D}$. $C$ is 
residually finite and the FICwF$^P$ is true for
the mapping torus $C\rtimes \l t\r$ for any action of the
infinite cyclic group $\l t\r$ on $C$.
\end{itemize} 

Assume that the vertex groups of $\cal G$ belong to $\cal D$.
Then the FICwF$^P$ is true for $\pi_1({\cal G})$ if ${\cal G}$ satisfies the
intersection property.
\end{thm}

Theorem \ref{ip} is a special case of Proposition \ref{residually-prop} $(2)$ 
and is proved in Subsection 3.7.

As a consequence of Theorem \ref{ip} we prove the following.

Let us first recall that a group $G$ is called {\it subgroup separable} 
if the following is satisfied. 
For any finitely generated subgroup $H$ of $G$ and $g\in G-H$  
there is a finite index normal subgroup 
$K$ of $G$ so that $H\subset K$ and $g\in G-K$. Equivalently, 
a group is subgroup separable if the finitely generated subgroups 
of $G$ are closed in the profinite topology of $G$. 
A subgroup separable group is 
therefore residually finite.

\begin{defn} \label{almostdefn} {\rm Let $\cal G$ be a graph 
of groups. An edge $e$ of $\cal G$ is called a {\it finite edge} 
if the edge group ${\cal G}_e$ is finite. $\cal G$ is called 
{\it almost a tree 
of groups} if there are finite edges $e_1,e_2,\ldots$ so that the  
components of ${\cal
G}-\{e_1,e_2,\ldots \}$ are trees. If we remove all the finite edges 
from a graph of groups then the components of the resulting 
graph are called {\it component subgraphs}.}\end{defn} 

\begin{thm} \label{introthm} Let $\cal G$ be a graph of groups. 
Then the FICwF$^P$ is true for $\pi_1({\cal G})$ if one of the
following five conditions is satisfied.

$(1).$ The vertex groups are virtually polycyclic and ${\cal G}$ satisfies
the intersection property.

$(2).$ The vertex groups of $\cal G$ are finitely generated nilpotent and 
$\pi_1({\cal G})$ is subgroup separable.

$(3).$ The vertex and the edge groups of any component subgraph 
(Definition \ref{almostdefn}) are fundamental
groups of closed surfaces of genus $\geq 2$. Given a component 
subgraph $\cal H$ which has at least one edge there is a subgroup 
$C<\cap_{e\in E_{\cal H}}{\cal H}_e$, 
which is of finite index in some edge group and is normal in $\pi_1({\cal H})$.

$(4).$ $\cal G$ is almost a tree of groups and the 
vertex groups of any component subgraph of $\cal G$   
are finitely generated abelian and of the same rank.

$(5).$ $\cal G$ is a tree of finitely generated abelian groups and either it has only one 
edge or FICwF$^P$ is true 
for any graph of infinite cyclic groups.\end{thm}

For examples of graphs of groups satisfying the hypothesis 
in $(3)$ see Example \ref{exm}. 

Using some recent results we deduce the following proposition.

\begin{prop}\label{k-l-theory} Let $\cal G$ be a tree of virtually cyclic groups.
Then the fibered isomorphism conjectures in $L$-theory and 
lower $K$-theory are true for $\pi_1({\cal G})\wr H$ for any finite 
group $H$.\end{prop}

\begin{rem}{\rm Although we stated the above theorems for the pseudo-isotopy version 
 of the conjecture, the same statements are also true in the $L$-theory version 
of the conjecture. See \cite{R5} and Remark \ref{lastremark}.}\end{rem}

We start by stating the general 
fibered isomorphism conjecture for equivariant homology theory (\cite{BL}). We 
show that under certain conditions a group acting on a tree satisfies this 
general conjecture provided the stabilizers also satisfy the 
conjecture. Then we restrict to the pseudo-isotopy case of the
conjecture and prove the theorems. 

Finally, in Examples \ref{end1}, \ref{end} and \ref{end0} 
we provide explicit examples of groups for which 
the results of this paper can be applied to 
prove the fibered isomorphism conjecture in the 
pseudo-isotopy and $L$-theory case. We further show that the 
groups in these examples are neither $CAT(0)$ nor hyperbolic. 

\subsection{Statements of the conjecture and some propositions}

We now recall the statement of the isomorphism 
conjecture for equivariant homology theories ([\cite{BL}, Section 1]) and 
we state some propositions. 

Let ${\cal H}^-_*$ be an equivariant homology theory with values in 
$R$-modules for $R$ a commutative associative ring with unit. 
An equivariant homology theory assigns to 
a group $G$ a $G$-homology theory ${\cal H}^G_*$, which for a pair 
of $G$-CW  complex $(X,A)$, produces a ${\Bbb Z}$-graded $R$-module 
${\cal H}^G_n(X,A)$. For details see [\cite{L1}, Section 1].

A {\it family} of subgroups of a group $G$ is defined as a set   
of subgroups of $G$, which is closed under taking subgroups and 
conjugations. If $\cal C$ is a class of 
groups, which is closed under isomorphisms and taking subgroups then we 
denote by ${\cal C}(G)$ the set of all subgroups of $G$, which belong to 
$\cal C$. 
Then ${\cal C}(G)$ is a family of subgroups of $G$. For example $\cal 
{VC}$, the class of virtually cyclic groups, is closed under isomorphisms 
and taking subgroups. By definition a virtually cyclic group has a 
cyclic subgroup of finite index. Also ${\cal {FIN}}$, the class of 
finite groups is closed under isomorphisms and taking subgroups.
   
Given a group homomorphism $\phi:G\to H$ and a family $\cal C$ of 
subgroups of $H$ define $\phi^*{\cal C}$ to be the family 
of subgroups $\{K<G\ |\ \phi (K)\in {\cal C}\}$ of $G$. Given a family 
$\cal C$ of subgroups of a group $G$ there is a $G$-CW complex $E_{\cal 
C}(G)$, which is unique up to $G$-equivalence, satisfying the property that 
for $H\in {\cal C}$ the fixpoint set $E_{\cal C}(G)^H$ is 
contractible and $E_{\cal C}(G)^H=\emptyset$ for $H$ not in ${\cal C}$. 

\begin{fjic}([\cite{BL}, Definition 1.1]) Let ${\cal H}^-_*$ be an 
equivariant homology theory with values in $R$-modules. Let $G$ be a group 
and $\cal C$ be a family of subgroups of $G$. Then the {\it 
isomorphism conjecture} for the pair $(G, {\cal C})$ states that the 
projection 
$p:E_{\cal C}(G)\to pt$ to the point $pt$ induces an isomorphism 
$${\cal H}^G_n(p):{\cal H}^G_n(E_{\cal C}(G))\simeq {\cal H}^G_n(pt)$$ for $n\in 
{\Bbb Z}$. 

The {\it fibered isomorphism  conjecture} for the pair $(G, {\cal 
C})$ states that for any group homomorphism $\phi: K\to G$ the 
isomorphism conjecture is true for the pair $(K, \phi^*{\cal C})$.\end{fjic}

Let $\cal C$ be a class of groups which is closed under isomorphisms and 
taking subgroups.

\begin{defn} \label{definition} {\rm If the 
(fibered) isomorphism conjecture is
true for the pair $(G, {\cal C}(G))$ we say that the 
{\it (F)IC$_{\cal C}$ is true for $G$} or simply say 
{\it (F)IC$_{\cal C}(G)$ is satisfied}. Also we say that the {\it 
(F)ICwF$_{\cal C}(G)$ is satisfied} if
the {\it (F)IC$_{\cal C}$} is true for $G\wr H$ for any finite group
$H$. Finally, a
group homomorphism $p:G\to K$ is said to {\it satisfy the FIC$_{\cal C}$}  or {\it the 
FICwF$_{\cal C}$} 
if for $H\in p^*{\cal C}(K)$ the FIC$_{\cal C}$ or the FICwF$_{\cal C}$ is true for $H$ respectively.}\end{defn} 

The (fibered) isomorphism conjectures in the pseudo-isotopy theory, 
$K$-theory and in the $L$-theory are equivalent to the Farrell-Jones conjectures 
stated in (\S 1.7) \S 1.6 in \cite{FJ}. (For details see [\cite{BL}, Sections 5 and 6] for
the $K$ and $L$ theories and see [\cite{LR}, 4.2.1 and 4.2.2] for the pseudo-isotopy theory.) 

\begin{defn} \label{property} {\rm We say that {\it ${\cal T}_{\cal C}$} 
({\it $_w{\cal T}_{\cal C}$}) is satisfied if for a
graph of groups $\cal G$ with vertex groups from the class $\cal C$ 
the FIC$_{\cal C}$ (FICwF$_{\cal
C}$) for $\pi_1({\cal G})$ is true. 

Let us now assume that $\cal C$ contains all the 
finite groups. We say that {\it $_t{\cal T}_{\cal C}$} 
({\it $_f{\cal T}_{\cal C}$})
is satisfied if for a graph of 
groups $\cal G$ with trivial (finite) edge
groups and vertex
groups belonging to the class $\cal C$, the FIC$_{\cal
C}$ for $\pi_1({\cal G})$ is true. If we replace the FIC$_{\cal
C}$ by the FICwF$_{\cal C}$  
then we denote the corresponding properties by {\it
$_{wt}{\cal T} _{\cal C}$} (
{\it $_{wf}{\cal T} _{\cal C}$}). Clearly 
$_{w}{\cal T} _{\cal C}$ implies 
${\cal T} _{\cal C}$ and $_{w*}{\cal T} _{\cal C}$ implies 
$_{*}{\cal T} _{\cal C}$ where $*=t$ or 
$f$.

And we say that {\it ${\cal P} _{\cal C}$} is 
satisfied if for $G_1, G_2\in {\cal C}$ the product
$G_1\times G_2$ satisfies the FIC$ _{\cal C}$.

{\it ${\cal {FP}} _{\cal C}$} is satisfied if 
whenever the FIC$ _{\cal C}$ is true for two groups 
$G_1$ and $G_2$ then the FIC$ _{\cal C}$ is true for 
the free product $G_1*G_2$.

We denote the above properties for the  
equivariant homology theories $P,K, L$ or $KH$  with only a 
super-script by $P,K, L$ or $KH$ respectively. For example, 
${\cal T}_{\cal C}$ 
for $P$ is denoted by ${\cal T}^P$ since in all the first three 
cases we set ${\cal C}={\cal {VC}}$ and for $KH$  we 
set ${\cal C}={\cal {FIN}}$.}\end{defn}

We show in this article and in \cite{R4} 
that the above properties are satisfied in several 
instances of the conjecture. 

For the rest of this section we assume that $\cal C$ is also closed under 
quotients and contains all the finite groups. 

\begin{prop} \label{theorem} {\bf (Graphs of groups).} 
Let ${\cal C}={\cal {FIN}}$ or $\cal {VC}$. Let $\cal G$ be a graph of
groups and assume there is a homomorphism 
$f:\pi_1({\cal G})\to Q$ onto another group $Q$ so that the restriction
of $f$ to any vertex group has finite kernel. 
If the FIC$ _{\cal C}(Q)$ 
(FICwF$ _{\cal C}(Q)$) is satisfied then 
the FIC$ _{\cal C}(\pi_1({\cal G}))$ 
(FICwF$ _{\cal C}(\pi_1({\cal G}))$) is 
also satisfied
provided one of the 
following holds.
(In the FICwF$ _{\cal C}$-case assume in addition that 
${\cal P} _{\cal C}$ is 
satisfied.)

$(a).$ ${\cal T} _{\cal C}$ 
($_w{\cal T} _{\cal C}$) is satisfied.

$(b).$ The edge groups of $\cal G$ are finite and 
$_f{\cal T} _{\cal C}$ 
($_{wf}{\cal T} _{\cal C}$) is satisfied.

$(c).$ The edge groups of $\cal G$ are trivial and 
$_t{\cal T} _{\cal C}$ 
($_{wt}{\cal T} _{\cal C}$) is satisfied.

\end{prop}

For the definition of continuous ${\cal H}^-_*$ in the following 
statement 
see [\cite{BL}, Definition 3.1]. See also Proposition \ref{proposition}.

\begin{prop} \label{residually-prop}{\bf (Graphs of residually finite
groups)} Assume that ${\cal P} _{\cal C}$ and $_{wt}{\cal T} _{\cal C}$ 
are satisfied. Let $\cal G$ be a graph of groups. If $\cal G$ is 
infinite then assume that ${\cal H}^-_*$ is continuous. 

$(1).$ Assume that the edge groups of $\cal G$ are finite and the vertex
groups are residually finite. If the FICwF$ _{\cal {C}}$ is
true for the vertex groups of $\cal G$, then it is true for $\pi_1({\cal G})$.

For the next three items assume that ${\cal C}={\cal {VC}}$.

$(2).$ Let $\cal D$ be the collection of groups defined in Theorem
\ref{ip} replacing\\ `FICwF$^P$' by `FICwF$ _{\cal
{VC}}$'. Assume that the vertex groups of $\cal G$ belong to $\cal D$. Then the 
FICwF$ _{\cal {VC}}$ is true for $\pi_1({\cal G})$ if $\cal
G$ satisfies the intersection property.

$(3).$ Assume that the vertex groups of $\cal G$ are virtually polycyclic
and that the FICwF$ _{\cal {VC}}$ is true for virtually
polycyclic groups. Then the FICwF$ _{\cal {VC}}$ is true for
$\pi_1({\cal G})$ provided either $\cal G$ satisfies the
intersection property or the vertex groups are finitely generated
nilpotent and $\pi_1({\cal G})$ is subgroup separable.

$(4).$ Assume that the vertex 
and the edge groups of
any component subgraph (Definition \ref{almostdefn}) are
fundamental
groups of closed surfaces of genus $\geq 2$ and for every component 
subgraph $\cal H$, which has at least one edge there is a subgroup 
$C<\cap_{e\in {\cal H}}{\cal H}_e$ 
which is of finite index in some edge group and is normal in $\pi_1({\cal H})$.
Then 
the FICwF$ _{\cal {VC}}$ is true for $\pi_1({\cal G})$
provided the FICwF$ _{\cal {VC}}$ is true for the
fundamental groups of closed $3$-manifolds which fiber over the circle.
\end{prop}

We denote a countable infinitely generated free group by 
$F^{\infty}$ and a countable infinitely generated abelian 
group by ${\Bbb Z}^{\infty}$. Also 
$G\rtimes H$ is a semidirect product with 
respect to an arbitrary action of $H$ on $G$. When 
$H$ is infinite cyclic and generated by the symbol $t$; 
we denote it by $\l t\r$.

\begin{prop} \label{maintheorem1} {\bf (Graphs of abelian groups).} 
Let $\cal G$ be a graph of groups whose vertex groups are 
abelian and let ${\cal H}^-_*$ be continuous. 

$(1).$ Assume that the FIC$ _{\cal
{VC}}$ is true for $F^{\infty}\rtimes \l
t\r$ and that ${\cal P} _{\cal
{VC}}$ is satisfied.
Then the FIC$_{\cal 
{VC}}$ is true for $\pi_1({\cal G})$ provided one of the 
following holds.

(a). $\cal G$ is a tree and the vertex groups of 
$\cal G$ are finitely generated and torsion free.  

For the next two items assume that the FICwF$ _{\cal
{VC}}$ is true for $F^{\infty}\rtimes \l
t\r$.

(b). $\cal G$ is a tree.

(c). $\cal G$ is not a tree and the FIC$ _{\cal 
{VC}}$ is true for ${\Bbb Z}^{\infty}\rtimes \l t\r$ for any countable 
infinitely generated abelian group ${\Bbb Z}^{\infty}$. 

$(2).$ Assume that ${\cal P} _{\cal {VC}}$ and 
$_{wt}{\cal T} _{\cal {VC}}$ are satisfied. 
Furthermore, assume that $\cal G$ is almost a tree of groups and the
vertex and the edge groups of any component subgraph of $\cal G$ 
are finitely generated and have the same rank. Then the FICwF$_{\cal
{VC}}$ is true for $\pi_1({\cal G})$ provided one of the followings 
is satisfied.

(i). The FICwF$ _{\cal
{VC}}$ is true for ${\Bbb Z}^n\rtimes \l t\r$ for all positive integers  
$n$.

(ii). The vertex and the edge groups of any component subgraph of 
$\cal G$ have rank equal to $1$.

$(3).$ Assume that ${\cal P} _{\cal {VC}}$ and 
$_{w}{\cal T} _{\cal {VC}}$ (${\cal T}_{\cal {VC}}$) are satisfied. 
Furthermore, assume that $\cal G$ is a tree of groups. Then the 
FICwF$_{\cal {VC}}$ (FIC$_{\cal {VC}}$) is true for $\pi_1({\cal G})$. 
\end{prop}

In $(1)$ of Proposition \ref{maintheorem1}
 if we assume that the  FICwF$ _{\cal 
{VC}}$ is true for ${\Bbb Z}^{\infty}\rtimes \l t\r$ and for 
$F^{\infty}\rtimes \l t\r$ then one can deduce from the same 
proof, using $(3)$ of Proposition \ref{product} instead 
of Lemma \ref{inverse}, that the  FICwF$ _{\cal
{VC}}$ is true for $\pi_1({\cal G})$ irrespective of whether 
$\cal G$ is a tree or not.

\begin{prop} \label{TC} {\bf ($_w{\cal T}^{P}$).} Let $\cal G$ be a 
graph of virtually cyclic groups so that the graph is almost a tree. 
Furthermore, assume that either the edge groups are finite or the 
infinite vertex groups are abelian.
Then the FICwF$^{P}$ is true for 
$\pi_1({\cal G})$.\end{prop}

An immediate corollary of the proposition is the following.

\begin{cor} \label{tt} {\bf ($_{wf}{\cal T}^{P}$).} $_{wf}{\cal T}^{P}$ 
(and hence $_{wt}{\cal T}^{P}$) is satisfied.\end{cor}

\begin{rem}\label{possible}{\rm We remark here that it 
is not yet known whether the FIC$^{P}$ is true 
for the $HNN$-extension $G=C*_C$, with 
respect to the maps $id:C\to C$ and $f:C\to C$ where $C$ is an infinite 
cyclic group, $id$ is the identity map and $f(u)=u^2$ for $u\in C$. This 
was mentioned in the introduction of \cite{FL}. Note 
that $G$ is isomorphic to the semidirect product ${\Bbb 
Z}[\frac{1}{2}]\rtimes \l t\r$, where $t$ acts on ${\Bbb 
Z}[\frac{1}{2}]$ by multiplication by $2$. The main problem with this
example is that ${\Bbb 
Z}[\frac{1}{2}]\rtimes \l t\r$ is not subgroup separable.}\end{rem}

Now, we recall that a property ({\it tree property}) similar to ${\cal 
T} _{\cal C}$ was defined in 
[\cite{BL}, Definition 4.1]. The 
tree property of \cite{BL} is stronger than ${\cal T}_{\cal 
C}$. Corollary 4.4 in \cite{BL} was proved under the assumption 
that the tree property is satisfied. In the following proposition we state  
that this is true with the weaker assumption that ${\cal 
T} _{\cal 
{FIN}}$ is satisfied. The proof of the proposition goes exactly in 
the same way as the proof of [\cite{BL}, Corollary 4.4]. It can also be deduced
from Proposition \ref{theorem}.

\begin{prop} \label{finite} Let 
$1\to K\to G\to Q\to 1$ be an exact sequence of groups. Assume that  
${\cal T} _{\cal {FIN}}$ is satisfied and $K$ acts 
on a tree with 
finite stabilizers and that the FIC$ _{\cal {FIN}}(Q)$ is 
satisfied. Then the FIC$ _{\cal {FIN}}(G)$ is also 
satisfied.\end{prop} 

\subsection{Graphs of groups}

In this subsection, we prove some results on graphs of groups needed for the 
proofs of the theorems and propositions.

We start by recalling that by Bass-Serre theory,  
a group acts on a tree without inversion if and only 
if the group is isomorphic to the 
fundamental group of a graph of 
groups (the Structure Theorem I.4.1 in \cite{DD}). Therefore,  
throughout the paper, by `action of a group on a tree' we will mean 
an action without inversion.

\begin{lemma} \label{almost} Let $\cal G$ be a finite and almost a tree
of groups (Definition \ref{almostdefn}). 
Then there is another graph of groups $\cal H$ with the following 
properties.

$(1).$ $\pi_1({\cal G})\simeq \pi_1({\cal H})$.

$(2).$ Either $\cal H$ has no edge or the edge groups 
of $\cal H$ are finite.

$(3).$ The vertex groups of $\cal H$ are of the form $\pi_1({\cal K})$ 
where $\cal K$ varies over subgraphs of groups of $\cal G$ which are 
maximal with respect to the 
property that the underlying graph of $\cal K$ is a tree and the 
edge (if there is any) groups of 
$\cal K$ are all infinite.\end{lemma}

\begin{proof} 
The proof is by
induction on the number of edges of the graph. Recall 
that an edge is called a finite edge if the corresponding edge 
group is finite (Definition \ref{almostdefn}).
If the graph 
has no finite edge then by definition 
of almost a graph of groups $\cal G$ is a tree. In this case we can take 
$\cal H$ to be a graph consisting of a single vertex and associate the 
group $\pi_1({\cal G})$ to the vertex. So assume 
${\cal G}$ has
$n$ finite edges and that the lemma is true for graphs with $\leq 
n-1$ finite edges.   
Let $e$ be an edge of ${\cal G}$ with ${\cal G}_e$ finite. 
If ${\cal G}-\{e\}$ is connected
then $\pi_1({\cal G})\simeq \pi_1({\cal
G}_1)*_{{\cal G}_e}$ where ${\cal G}_1={\cal G}-\{e\}$ is a graph 
with $n-1$ finite 
edges. By the induction hypothesis there 
is a graph of groups ${\cal H}_1$ satisfying $(1), (2)$ and $(3)$ for 
${\cal G}_1$. Let $v_1$ and $v_2$ be the vertices of $\cal G$ to  
which the ends of $e$ are 
attached and let $v'_1$ and $v'_2$ be the vertices of 
${\cal H}_1$ so that ${\cal G}_{v_i}$ is a 
subgroup of ${\cal H}_{v'_i}$ for $i=1,2$. (Note 
that $v_1$ and $v_2$ could be the same vertex). 
Define $\cal H$ by attaching an edge 
$e'$ to ${\cal H}_1$ so that the ends of $e'$ are attached to 
$v'_1$ and $v'_2$ and associate 
the group ${\cal G}_e$ to $e'$. The injective homomorphisms  
${\cal H}_{e'}\to {\cal H}_{v'_i}$ for $i=1,2$   
are defined by the homomorphisms ${\cal G}_e\to {\cal G}_{v_i}$.  It is now 
easy to check that $\cal H$ satisfies $(1), (2)$ and $(3)$ for $\cal G$.
On the other hand if ${\cal G}-\{e\}$ has two components say ${\cal G}_1$ and 
${\cal G}_2$ then $\pi_1({\cal G})\simeq \pi_1({\cal
G}_1)*_{{\cal G}_e} \pi_1({\cal G}_2)$ where ${\cal G}_1$ and ${\cal G}_2$ 
have $\leq
n-1$ finite edges. Using the induction hypothesis and 
a similar argument as above we complete the proof of the  
lemma.\end{proof}

\begin{lemma}\label{finitefree} A finitely generated group contains a free 
subgroup of finite 
index if and only if the group acts on a tree with finite 
stabilizers. And a group acts on a tree with trivial 
stabilizers if and only if the group is free.\end{lemma}

\begin{proof} By  
[\cite{DD}, Theorem IV.1.6] a group contains a free subgroup of finite index if 
and only if the group acts on a tree with finite stabilizers 
and the stabilizers have bounded order. The proof of the lemma 
now follows easily.\end{proof}

We will also need the following two Lemmas.

\begin{lemma} \label{stark} Let $\cal G$ be a graph of finitely 
generated abelian groups so that the underlying graph of $\cal G$ is a 
tree. Then the restriction of the abelianization 
homomorphism $\pi_1({\cal G})\to H_1(\pi_1({\cal G}), {\Bbb Z})$ to each 
vertex group of 
the tree of groups $\cal G$ is injective.\end{lemma}

\begin{proof} If the tree $\cal G$ is finite and the vertex groups are 
finitely generated free abelian then it was proved in 
[\cite{S}, Lemma 3.1] that there is a homomorphism $\pi_1({\cal 
G})\to A$ onto a free abelian group so that the restriction of this 
homomorphism to each vertex group is injective. In fact it was shown  
there that the abelianization homomorphism $\pi_1({\cal 
G})\to H_1(\pi_1({\cal G}), {\Bbb Z})$ is injective when restricted 
to the vertex groups,  
and, then since the vertex groups are torsion free 
$\pi_1({\cal
G})\to A=H_1(\pi_1({\cal 
G}), {\Bbb Z})/\{torsion\}$ is also injective on the vertex groups. 
The same proof 
goes through without the torsion free vertex group assumption to 
prove the lemma when $\cal G$ is finite. Therefore, we 
mention 
the additional arguments needed in the infinite case. First, write 
${\cal G}$ as 
an increasing union of finite trees ${\cal G}_i$ of finitely 
generated abelian groups. Consider the following commutative diagram.

$$\begin{CD}
\pi_1({\cal G}_i) @>>> H_1(\pi_1({\cal 
G}_i), {\Bbb Z})\\
@VVV    @VVV\\
\pi_1({\cal G}_{i+1}) @>>> H_1(\pi_1({\cal 
G}_{i+1}), {\Bbb Z})
\end{CD}$$

Note that the left hand side vertical map is injective. By the finite 
tree case the restriction of the two horizontal maps to each vertex 
group of the respective trees are injective.
Now, since group homology and fundamental group commute 
with direct limit, taking limit completes the proof 
of the Lemma.\end{proof}

\begin{lemma} \label{fi} Let $\cal G$ be a finite graph of finitely 
generated groups
satisfying the following.

{\bf P.} Each edge group 
is of finite index in the end vertex groups of the
edge. Also assume that the intersection of the edge 
groups contains a subgroup $C$ (say), which is normal 
in $\pi_1({\cal G})$ and is of finite index in some edge $e$ (say) group. 

Then $\pi_1({\cal
G})/C$ is isomorphic to the fundamental group of a  
finite-type (Definition \ref{def1.1}) 
graph of groups. Consequently $\cal G$ has the
intersection property.\end{lemma} 

\begin{proof} 
The proof is by induction on the number of edges of the graph. 

{\bf Induction hypothesis.} $(IH_n)$ For any finite graph of groups $\cal 
G$ with 
$\leq n$ edges, which satisfies {\bf P}, $\pi_1({\cal
G})/C$ is isomorphic to the fundamental group of a graph of groups whose
underlying graph is the same as that of $\cal G$ and the vertex 
groups are finite and isomorphic to ${\cal G}_v/C$ where 
$v\in V_{\cal G}$. 

If $\cal G$ has one edge $e$ then $C<{\cal G}_e$ and $C$ 
is normal in $\pi_1({\cal G})$. First, assume $e$ disconnects 
$\cal G$. Since ${\cal G}_e$ is of 
finite index in the end vertex groups and $C$ is also 
of finite index in ${\cal G}_e$, $C$ is of finite index in the 
end vertex groups of $e$. Therefore, $\pi_1({\cal G})/C$ 
has the desired property. The argument is the same when 
$e$ does not disconnect $\cal G$.
 
Now, assume $\cal G$ has $n$ edges and satisfies {\bf P}.

Let us first consider the case where there is an 
edge $e'$ other than $e$ so that  
${\cal G}-\{e'\}={\cal D}$ (say)  
is a connected graph. Note that $\cal D$ has $n-1$ edges 
and satisfies $\bf P$.

Hence, by $IH_{n-1}$, $\pi_1({\cal D})/C$ is isomorphic 
to the fundamental group of a finite-type 
graph of groups with $\cal D$ as the underlying 
graph and the vertex groups 
are of the form ${\cal D}_v/C$.  

Let $v$ be an end vertex of $e'$. Then 
${\cal D}_v/C$ is finite. 
Also by hypothesis ${\cal G}_{e'}$ 
is of finite index in ${\cal G}_v={\cal D}_v$. Therefore, 
${\cal G}_{e'}/C$ is also finite. This 
completes the proof in this case.

Now, if every edge $e'\neq e$ disconnects $\cal G$ then $\cal G-\{e\}$ is a 
tree. Let $e'$ be an edge other than $e$ so that one 
end vertex $v'$ (say) of $e'$ has valency $1$. Such an edge 
exists because $\cal G-\{e\}$ is a tree. Let 
${\cal D}={\cal G}-(\{e'\}\cup \{v'\})$. 
Then 
$\pi_1({\cal G})\simeq \pi_1({\cal D})*_{{\cal G}_{e'}}{\cal G}_{v'}$. 
Let $v''$ be the other end vertex of $e'$. Then by the induction 
hypothesis $C$ is of finite index in 
${\cal G}_{v''}={\cal D}_{v''}$. Hence, $C$ 
is of finite index in ${\cal G}_{e'}$. Also ${\cal G}_{e'}$ is of 
finite index in ${\cal G}_{v'}$. Therefore, $C$ is also of finite index 
in ${\cal G}_{v'}$. 

This completes the proof.  
\end{proof}

The following lemma and example give some concrete examples of graphs of groups
with intersection property.

\begin{lemma} \label{ic} Let $\cal G$ be a finite graph of groups so that 
all the vertex and the edge groups are finitely generated abelian and of the 
same rank $r$ (say) and the underlying graph of $\cal G$ is a tree. Then 
$\cap_{e\in E_{\cal G}}{\cal G}_e$ contains a rank $r$ free 
abelian subgroup 
$C$ which is normal in $\pi_1({\cal G})$ so that $\pi_1({\cal
G})/C$ is isomorphic to the fundamental group of a graph of groups whose 
underlying graph is $\cal G$ and vertex groups are finite and isomorphic 
to ${\cal G}_v/C$ where $v\in V_{\cal G}$.\end{lemma}

\begin{proof} The proof is by induction on the number of edges. 
If the graph has one edge $e$ then clearly ${\cal G}_e$ is normal 
in $\pi_1({\cal G})$, for ${\cal G}_e$ is normal in the two end 
vertex groups. This follows from Lemma \ref{gene}. 
So by induction assume that the lemma is true for 
graphs with $\leq n-1$ edges. Let $\cal G$ be a finite graph with $n$ 
edges and satisfies the hypothesis of the Lemma. Consider an 
edge $e$ which has one end vertex $v$ (say) with valency $1$. Such 
an edge exists because the graph is a tree. Let 
$v_1$ be the other end vertex of $e$. Then 
$\pi_1({\cal G})\simeq \pi_1({\cal G}')*_{{\cal G}_e}{\cal G}_v$. 
Here ${\cal G}'={\cal G}-(\{e\}\cup\{v\})$. Clearly by the induction hypothesis 
there is a finitely generated free abelian normal subgroup 
$C_1<\cap_{e'\in E_{{\cal G}'}}{{\cal G}'}_{e'}$ of rank 
$r$ of $\pi_1({\cal G}_1)$ satisfying all the required properties. 
Now, note that $C_1\cap {\cal G}_e=C'$ (say) is of finite index 
in ${{\cal G}'}_{v_1}$ and also in $C_1$ and $C'$ has rank $r$. 
This follows from the following easy to 
verify Lemma. Now, since $C_1$ is finitely generated and $C'$ is of finite 
index in $C_1$ we can find a characteristic subgroup $C (<C')$ of $C_1$ of finite 
index. Therefore, $C$ has rank $r$ and is normal in $\pi_1({\cal G}')$, since 
$C_1$ is normal in $\pi_1({\cal G}')$. Obviously $C$ is normal in 
${\cal G}_v$. Therefore, we again use Lemma \ref{gene} 
to conclude that $C$ is normal in 
$\pi_1({\cal G})$. The other properties are clearly satisfied.
This completes the proof of the Lemma.

\begin{lemma} \label{rank} Let $G$ be a finitely generated abelian group 
of rank $r$ and let $G_1,\ldots , G_k$ be rank $r$ subgroups of $G$. Then 
$\cap_{i=1}^{i=k}G_i$ is of rank $r$ and of finite index in 
$G$.\end{lemma}  
\end{proof}

\begin{exm}\label{exm} {\rm Let $G$ and $H$ be finitely generated groups 
and let $H$ be a finite index normal subgroup of $G$. Let $f:H\to G$ 
be the inclusion. Consider a finite tree of groups $\cal G$ whose vertex groups 
are copies of $G$ and the edge groups are copies of $H$. Also assume 
that the maps from the edge groups to the vertex groups (defining the 
tree of group structure) are $f$. Then $\cal G$ has the intersection 
property.}\end{exm}

\begin{lemma} \label{gene} Let $G=G_1*_HG_2$ be a generalized 
free product. If $H$ is normal 
in both $G_1$ and $G_2$ then $H$ is normal in $G$.\end{lemma}

\begin{proof} The proof follows by using the normal form of 
elements in a generalized free product. See [\cite{LS}, p. 72].
\end{proof}

\subsection{Residually finite groups} 
We now recall and also prove some basic results we need on
residually finite groups. For this section we 
abbreviate `residually finite' by $\cal {RF}$.

\begin{lemma} \label{graph-res} The
fundamental group of a finite graph of $\cal {RF}$ groups with finite edge
groups is $\cal {RF}$.\end{lemma}

\begin{proof} The proof is by induction on the number of edges. If there
is no edge then there is nothing to prove. So assume the lemma for graphs
with $\leq n-1$ edges. Let $\cal G$ be a graph of groups with $n$ edges and 
satisfying the hypothesis. It follows that $\pi_1({\cal G})\simeq
\pi_1({\cal G}_1)*_F\pi_1({\cal G}_2)$ or $\pi_1({\cal G})\simeq
\pi_1({\cal G}_1)*_F$ where $F$ is a finite group and ${\cal G}_i$ satisfy
the hypothesis of the lemma and have $\leq n-1$ edges. Also note that
$\pi_1({\cal G}_1)*_F\pi_1({\cal G}_2)$ can be embedded as a 
subgroup in $(\pi_1({\cal
G}_1)*\pi_1({\cal G}_2))*_F$. Therefore, by the induction hypothesis and
since a 
free product of $\cal {RF}$ groups is again
$\cal {RF}$ (\cite{Gr}) we only need to prove that, for
finite $H$, $G*_H$
is $\cal {RF}$ if so is $G$. But, this follows from
\cite{BT} or  [\cite{DC}, Theorem 2]. This completes the proof 
of the Lemma.\end{proof}

\begin{lemma} \label{res-ext} Let $1\to K\to G\to H\to 1$ be an extension
of groups so that $K$ and $H$ are ${\cal {RF}}$. Assume that any finite
index
subgroup of $K$ contains a subgroup $K'$ so that $K'$ is
normal in $G$ and $G/K'$ is $\cal {RF}$. 
Then $G$ is $\cal {RF}$.\end{lemma}

\begin{proof} Let $g\in G-K$ and $g'\in H$ be the image of $g$ in $H$.
Since $H$ is ${\cal {RF}}$ there is a finite index subgroup $H'$ of $H$
not
containing $g'$. The inverse image of $H'$ in $G$ is a finite index
subgroup not containing $g$. 

Next, let $g\in K-\{1\}$. Choose a finite index subgroup of $K$ not containing
$g$. By hypothesis there is a finite index subgroup $K'$ of $K$, which is
normal in $G$ and does not contain $g$. Also $G/K'$ is ${\cal {RF}}$. Now, 
applying the previous case we complete the proof.\end{proof}

\begin{lemma} \label{inter} Let $\cal G$ be a finite graph of 
finitely generated $\cal {RF}$  
groups satisfying the intersection property. Assume the following. 
`Given an edge $e$ and an end vertex $v$ of $e$, 
for every subgroup $E$ of ${\cal G}_e$, 
which is normal in ${\cal G}_v$, the quotient ${\cal G}_v/E$ is again $\cal {RF}$.' 
Then $\pi_1({\cal
G})$ is $\cal {RF}$.\end{lemma}

\begin{proof} Using Lemma \ref{graph-res} we can assume that all the edge
groups of $\cal G$ are infinite. Now, the proof is by induction on the number
of edges of the graph $\cal G$. Clearly the induction starts because 
if there is no edge then the lemma is true.
Assume the result for all graphs satisfying the hypothesis with
number of edges $\leq n-1$ and let $\cal G$ be a graph of groups with $n$
number of edges and satisfying the hypothesis of the lemma. By the 
intersection property 
there is a normal subgroup $K$, contained 
in all the edge groups, of $\pi_1({\cal G})$, which is of finite index in
some edge group. Hence, we have the following. 

{\bf K.} 
$\pi_1({\cal G})/K$ is isomorphic to the fundamental group of a finite
graph of $\cal {RF}$ groups (by hypothesis the quotient of a vertex group 
by $K$ is $\cal {RF}$) with $n$ edges and some edge 
group is finite. Also it is easily seen that this 
quotient graph of groups has the intersection property. 

By the induction hypothesis and by Lemma \ref{graph-res} $\pi_1({\cal
G})/K$ is $\cal {RF}$. 

Now, we would like to apply Lemma \ref{res-ext} to the exact sequence

$$1\to K\to \pi_1({\cal G})\to \pi_1({\cal G})/K\to 1.$$

Let $H$ be a finite index subgroup of $K$. Since $K$ is finitely
generated (being of finite index in a finitely generated group) we can
find a finite index characteristic subgroup $H'$ of $K$ contained in $H$.
Hence, $H'$ is normal in $\pi_1({\cal G})$. It
 is now easy to see that {\bf K} is satisfied if we replace $K$ by $H'$.
Hence, $\pi_1({\cal G})/H'$ is $\cal {RF}$.

Therefore, by Lemma \ref{res-ext} $\pi_1({\cal G})$ is $\cal {RF}$.
\end{proof}

\begin{lemma} \label{rfcyclic} Let $\cal G$ be a finite graph of
virtually cyclic groups
so that either the edge groups are finite or the 
infinite vertex groups are abelian and 
the associated graph is almost a tree. Then $\pi_1({\cal 
G})$ is ${\cal {RF}}$.\end{lemma}

\begin{proof} Applying Lemma \ref{graph-res} and using the 
definition of almost a graph of groups we can assume 
that the graph is a tree and all the edge groups are infinite. 
Now, using Lemma \ref{ic} we see that the hypothesis of Lemma \ref{inter} is
satisfied. This proves the Lemma.\end{proof}  

\begin{lemma} \label{res-surface} Let $\cal G$ be a finite graph of groups
whose 
vertex and edge groups are fundamental groups of closed 
surfaces of genus $\geq 2$. Also 
assume that the intersection of the edge groups contains a subgroup 
$C$ (say), which 
is normal in $\pi_1({\cal G})$ and is of finite index in some edge group. Then 
$\pi_1({\cal G})$ contains a normal subgroup isomorphic to 
the fundamental group of a closed surface so that the
quotient is isomorphic to the fundamental group of a finite-type  
graph of 
groups. Also
$\pi_1({\cal G})$ is $\cal {RF}$.\end{lemma}

\begin{proof} We need the
following Lemma.

\begin{lemma} Let $S$ be a closed surface.
Let $G$ be a subgroup of $\pi_1(S)$. Then $G$ is isomorphic to the
fundamental group of a closed surface if and only if $G$ is of 
finite index in $\pi_1(S)$.\end{lemma}

\begin{proof} The proof follows from covering space theory.\end{proof}

Therefore, using the above lemma we get that the edge groups 
of $\cal G$ are of finite index in the
end vertex groups of the corresponding edges. Hence, by 
Lemma \ref{fi} $\pi_1({\cal G})/C$
is the fundamental group of a finite-type graph of groups.
This proves the first statement. 

Now, 
by Lemma \ref{graph-res} $\pi_1({\cal G})/C$ is $\cal {RF}$. Next, by
\cite{GB1}, closed surface groups are $\cal {RF}$. Then using the 
above lemma, it is
easy to check that the hypothesis of Lemma \ref{res-ext} is satisfied for
the following exact sequence. $$1\to C\to \pi_1({\cal G})\to \pi_1({\cal
G})/C\to 1.$$

Hence, $\pi_1({\cal G})$ is $\cal {RF}$.
\end{proof}

\subsection{Basic results on the isomorphism conjecture}

$\cal C$ is always a class of groups closed under isomorphisms and 
taking subgroups unless otherwise mentioned.

We start by noting that if the FIC$ _{\cal C}$ is 
true for a group $G$ then the FIC$ _{\cal C}$ is also true 
for any subgroup $H$ 
of $G$. We will refer to this fact as the {\it hereditary property} in this 
paper.

By the Algebraic lemma in \cite{FR} if $G$ is a 
normal subgroup of $K$ then $K$ can be embedded in 
the wreath product $G\wr (K/G)$. We will be using this fact throughout the 
paper without explicitly mentioning it. 

\begin{lemma} \label{finiteindex} If the FICwF$ _{\cal 
C}(G)$ is satisfied 
then the FICwF$ _{\cal C}(L)$ is also satisfied for any 
subgroup $L$ of $G$. 
\end{lemma} 

\begin{proof} Note that given a group $H$, $L\wr H$ is a subgroup of $G\wr H$. Now, use the 
hereditary property of the FIC$ _{\cal
C}$.
\end{proof}

\begin{prop} \label{proposition} Assume that 
${\cal H}^-_*$ is continuous. Let $G$ be a group and 
$G=\cup_{i\in I}G_i$ where the $G_i$'s are an increasing sequence of subgroups of 
$G$ so that the FIC$ _{\cal C}(G_i)$ is satisfied for  
$i\in I$. Then the FIC$ _{\cal C}(G)$ is also 
satisfied. And if the FICwF$ _{\cal C}(G_i)$ is 
satisfied for 
$i\in I$ then the FICwF$ _{\cal C}(G)$ is also
satisfied.\end{prop}

\begin{proof} The first assertion is the same as the conclusion of 
[\cite{BL}, Proposition 3.4]  
and the second one is easily deducible from it, since given a group $H$, 
$G\wr H=\cup_{i\in I}(G_i\wr H)$.\end{proof}

\begin{rem} \label{finitegraph} {\rm 
Since the fundamental group of an infinite graph of groups can be written 
as an increasing union of fundamental groups of finite subgraphs, 
throughout rest of the paper we consider only finite graphs. The infinite case 
will then follow if the corresponding equivariant homology theory satisfies  
the assumption of Proposition \ref{proposition}. Examples of such 
equivariant homology theories appear in the 
isomorphism conjecture for pseudo-isotopy theory, $L$-theory and 
for $K$-theory. See [\cite{FL}, Theorem 7.1].}\end{rem} 

The following lemma from \cite{BL} is crucial for proofs of the
results in this paper. This result in the context of the 
original fibered isomorphism  conjecture (\cite{FJ}) 
was proved in [\cite{FJ}, Proposition 2.2].

\begin{lemma}\label{inverse} ([\cite{BL}, Lemma 2.5]) Let $\cal C$ be also 
closed under taking quotients. Let $p:G\to Q$ 
be a surjective group homomorphism and assume that 
the FIC$_{\cal C}$ is true for $Q$ and for $p$. 
Then $G$ 
satisfies the FIC$ _{\cal C}$.\end{lemma} 

\begin{prop} \label{product} Let $\cal C$ be as in the 
statement of the above lemma. Assume that ${\cal P}_{\cal 
C}$ is 
satisfied. 

(1). If the FIC$ _{\cal C}$ is true 
for $G_1$ and $G_2$ then the FIC$ _{\cal C}$ is true for $G_1\times G_2$. 
And, the same statement also holds for 
the FICwF$_{\cal C}$. 

(2). Let $G$ be a finite index normal subgroup of a group $K$. If the  
FICwF$_{\cal C}(G)$ is satisfied then the FICwF$ _{\cal 
C}(K)$ is also satisfied.

(3). Let $p:G\to Q$ be a group homomorphism. If 
the FICwF$_{\cal C}$ is true for $Q$ and for $p$ 
then the FICwF$_{\cal C}$ is true for $G$. 
\end{prop}

\begin{proof} The proof of $(1)$ is essentially two applications of Lemma 
\ref{inverse}. First, apply it to the projection $G_1\times G_2\to G_2$. 
Hence, to prove the lemma we need to check that the 
FIC$_{\cal C}$ is true for $G_1\times H$ for 
any $H\in {\cal C}(G_2)$. Now, fix $H\in {\cal C}(G_2)$ and apply Lemma 
\ref{inverse} to the projection $G_1\times H\to G_1$. Thus we need to 
show that the FIC$_{\cal C}$ is true for 
$K\times H$ where $K\in {\cal C}(G_1)$. But this is 
exactly ${\cal P} _{\cal C}$, which is true by 
hypothesis. 

Next, note that given a group $H$, $(G_1\times G_2)\wr H$ is a 
subgroup of $(G_1\wr H)\times (G_2\wr H)$. Therefore, the 
FICwF$_{\cal C}$ is true for $G_1\times G_2$ if it is true for $G_1$ and $G_2$.

For $(2)$ let $H=K/G$. Then $K$ is a 
subgroup of $G\wr H$. Let $L$ be a finite group; then it is easy to check 
that $$K\wr L < (G\wr H)\wr L \simeq G^{H\times L}\rtimes (H\wr L) $$
$$< G^{H\times L}\wr (H\wr L) < \Pi_{{|H\times 
L|}-times}(G\wr (H\wr L)).$$

The isomorphism in the above display follows from [\cite{FL}, Lemma 2.5] with respect to 
some action of $H\wr L$ on $G^{H\times L}$ (replace $X$ by $A$ and 
$Y$ by $B$ in [\cite{FL}, Lemma 2.5]). The second inclusion follows from 
the Algebraic Lemma in \cite{FR}.

Now, using $(1)$ and by hypothesis we complete the proof.

For $(3)$ we need to prove that the 
FIC$_{\cal C}$ is true for $G\wr H$ for any finite group $H$. We now 
apply Lemma \ref{inverse} to the homomorphism $G\wr H\to Q\wr H$. By 
hypothesis the FIC$ _{\cal C}$ is true for $Q\wr H$. So let 
$S\in {\cal C}(Q\wr H)$. We have to prove that the 
FIC$ _{\cal C}$ is true for $p^{-1}(S)$. 
Note that $p^{-1}(S)$ contains 
$p^{-1}(S\cap 
Q^H)$ as a normal subgroup of finite index. Therefore, using $(2)$ it is enough to 
prove the FICwF$ _{\cal C}$ for $p^{-1}(S\cap
Q^H)$. Next, note that $S\cap
Q^H$ is a subgroup of $\Pi_{h\in H}L_h$, where $L_h$ is the image 
of $S\cap Q^H$ under the projection to the $h$-th coordinate 
of $Q^H$, and since $\cal C$ is closed under taking quotient, 
$L_h\in {\cal C}(Q)$. Hence, 
$p^{-1}(S\cap Q^H)$ is a subgroup of $\Pi_{h\in H}p^{-1}(L_h)$. 
Since by hypothesis the 
FICwF$ _{\cal C}$ is true for $p^{-1}(L_h)$ 
for $h\in H$, using 
$(1)$, $(2)$ and 
Lemma \ref{finiteindex} we see 
that the FICwF$ _{\cal C}$ is true for $p^{-1}(S\cap
Q^H)$. This completes the 
proof.\end{proof} 

\begin{cor} \label{finabe} ${\cal P} _{\cal 
{VC}}$ implies that the FIC$ _{\cal {VC}}$ is true for 
finitely generated abelian groups.\end{cor}

\begin{proof} The proof is immediate from $(1)$ of Proposition 
\ref{product} since the FIC$ _{\cal {VC}}$ is true 
for virtually cyclic groups.\end{proof} 

\begin{rem} {\rm In $(2)$ of Proposition \ref{product}, if we assume 
that the FIC$ _{\cal C}(G)$ is satisfied instead of 
the FICwF$_{\cal C}(G)$, then it is not known 
how to deduce the FIC$_{\cal C}(K)$. Even in the case of the 
FIC$_{\cal {VC}}$ and when $G$ is a free group it is open. 
However if 
$G$ is free then the FIC$^{P}(K)$ is satisfied by  
results of Farrell-Jones. See the proof of  
Proposition \ref{graph-finite} for details. Also using a 
recent result (\cite{BL1}) it 
can be shown by the same method that the FIC$^L(K)$ is 
satisfied.}\end{rem}

\begin{prop} \label{integer} $_t{\cal T} _{\cal 
C}$ ($_{wt}{\cal T} _{\cal C}$) implies 
that the FIC$ _{\cal C}$ (FICwF$ _{\cal C}$) 
is true for any free group. And 
if $\cal C$ contains the class of all finite groups then 
$_f{\cal T}_{\cal C}$ ($_{wf}{\cal T}_{\cal C}$) 
implies that the FIC$_{\cal C}$ (FICwF$_{\cal C}$) 
is true for a finitely generated group, 
which contains a free subgroup of finite index.\end{prop}

\begin{proof} The proof follows from Lemma \ref{finitefree}.\end{proof}

\begin{cor} \label{freeabelian} $_t{\cal T} _{\cal 
C}$ ($_{wt}{\cal T} _{\cal
C}$) and 
${\cal P} _{\cal C}$ imply that the 
FIC$ _{\cal C}$ (FICwF$ _{\cal C}$) is true for 
finitely generated free abelian groups.\end{cor}

\begin{proof} The proof is a combination of Proposition \ref{integer} and 
$(1)$ of Proposition \ref{product}.\end{proof}

Let $F^n$ be a finitely generated free group of rank $n$.

\begin{lemma} \label{infin} Assume that the FIC$ _{\cal 
C}$ is true for ${\Bbb Z}^{\infty}\rtimes \l t\r$ 
($F^{\infty}\rtimes \l t\r$) for any countable 
infinitely generated abelian group ${\Bbb Z}^{\infty}$, then the 
FIC$ _{\cal C}$ is true for 
${\Bbb Z}^n\rtimes \l t\r$ ($F^n\rtimes \l t\r$) 
for all 
$n\in {\Bbb N}$. Here all actions of $t$ on the corresponding groups 
are arbitrary. 

And the same holds if we replace the FIC$_{\cal C}$ by 
the FICwF$ _{\cal C}$.\end{lemma}

\begin{proof} The proof is an easy consequence of the 
hereditary property and Lemma \ref{finiteindex}. \end{proof}

\begin{lemma} \label{pc} If the FICwF$ _{\cal 
{VC}}$ is true for $F^{\infty}\rtimes \l t\r$ for any 
action of $t$ on $F^{\infty}$, then ${\cal P}_{\cal {VC}}$ is satisfied.\end{lemma}

\begin{proof} Note that ${\Bbb Z}\times {\Bbb Z}$ is a subgroup of
$F^{\infty}\rtimes_t \l
t\r$, where the suffix `t' denotes the trivial action of $\l t\r$ on
$F^{\infty}$. Hence, the
FICwF$ _{\cal {VC}}$ is true for 
${\Bbb Z}\times {\Bbb Z}$, that is, the FIC$ _{\cal {VC}}$
is true for $({\Bbb Z}\times {\Bbb Z})\wr F$ for any finite group $F$. On
the other hand, for two virtually cyclic groups $C_1$ and $C_2$, the
product $C_1\times C_2$ contains a finite 
index free abelian normal subgroup (say $H$) of rank $\leq 2$ (Lemma \ref{last}), and,
therefore, $C_1\times C_2$ is a subgroup of $H\wr F$ for a finite group
$F$. Using the hereditary property we conclude that 
${\cal P}_{\cal {VC}}$ is satisfied.\end{proof}

\begin{prop} \label{polycyclic} The 
FICwF$^{P}$ is true for any virtually polycyclic group.\end{prop}

\begin{proof} By [\cite{FJ}, Proposition 2.4] the FIC$^{P}$ is true 
for any virtually poly-infinite cyclic group. Also a 
polycyclic group is virtually poly-infinite cyclic. 
Now, it is easy to check that the wreath product of a virtually 
polycyclic group with a finite group is virtually poly-infinite cyclic. This 
completes the proof.\end{proof}

Since the product of two virtually cyclic groups is 
virtually polycyclic, an immediate corollary is the following. 
We give a proof of the corollary which is independent of 
Proposition \ref{polycyclic}.

\begin{cor}\label{PVC} ${\cal 
P}^{P}$ is satisfied. Also FICwF$^{P}$ is true for 
any virtually cyclic group.\end{cor}

\begin{proof}[Proof of Corollary \ref{PVC}] First, note that 
the FIC$^P$ is true for virtually cyclic groups. Hence, for the 
first part we only 
have to prove that the FIC$^P$ is true for $V_1\t V_2$ where 
$V_1$ and $V_2$ are two infinite virtually cyclic groups. 
Note that $V_1\t V_2$ contains a finite index free abelian normal subgroup, 
say $A$, on two generators. Therefore, $V_1\t V_2$ embeds in $A\wr H$ for 
some finite group $H$. Since $A$ is isomorphic to the fundamental group 
of a flat $2$-torus, FIC$^P$ is true for $A\wr H$. See Fact 3.1 and Theorem A 
in \cite{FR}. Therefore, FIC$^P$ is true for $V_1\t V_2$ by the hereditary 
property. This proves that ${\cal 
P}^{P}$ is satisfied.

The proof of the second part is similar since for any virtually 
cyclic group $V$ and for any finite group $H$, $V\wr H$ 
is either finite or embeds in a group of the type $A\wr H'$ for some finite 
group $H'$ and where $A$ is isomorphic to a free abelian group 
on $|H|$ number of generators and, therefore, $A$ is isomorphic to the 
fundamental group of a flat $|H|$-torus. Then we can again apply Fact 3.1 
and Theorem A from \cite{FR}.\end{proof}

We will also need the following proposition.

\begin{prop} \label{graph-finite} Let $\cal G$ be a graph of finite
groups. Then $\pi_1({\cal G})$ satisfies the FICwF$^{P}$.\end{prop}

\begin{proof} By Remark \ref{finitegraph} we can assume that the graph is
finite. Lemma \ref{finitefree} implies that we need to show that the
FICwF$^{P}$ is true for finitely generated groups, which contains a free
subgroup of finite index. Now, it is a formal consequence of results of
Farrell-Jones
that the FICwF$^{P}$ is true for a free
group. For details see [\cite{FL}, Lemma 2.4]. 
Also compare [\cite{FR}, Fact 3.1]. Next, since 
${\cal P}^{P}$ is satisfied (Corollary
\ref{PVC}) using $(2)$ of Proposition \ref{product} we complete the
proof. 
\end{proof}

Proposition \ref{graph-finite} was used in \cite{PLMP} to calculate the 
lower $K$-groups of virtually free groups.

\subsection{Proofs of the propositions}

Recall that $\cal C$ is always a class of groups which is closed 
under isomorphism, taking subgroups and quotients and contains all the
finite groups.

\begin{proof} [Proof of Proposition \ref{maintheorem1}] 
Proof of $(1)$. Since ${\cal H}^-_*$ is continuous by Remark 
\ref{finitegraph} we can assume that the graph $\cal G$ is finite.

$1(a)$. Since the graph $\cal G$ is a 
tree and the vertex groups 
are torsion free, we can apply Lemma \ref{stark} to see that the
restriction of the homomorphism 
$f:\pi_1({\cal G})\to H_1(\pi_1({\cal G}), {\Bbb Z})/\{torsion\}=A (say)$ 
to each vertex group is injective. Let $K$ be the kernel 
of $f$. Let $T$ be the tree on which $\pi_1({\cal G})$ acts for the 
graph of group structure $\cal G$. Then $K$ also acts on $T$ 
with vertex stabilizers $K\cap g{\cal G}_v g^{-1}=(1)$ where $g\in 
\pi_1({\cal G})$ and $v\in V_{\cal G}$. Hence, by Lemma \ref{finitefree} $K$ is a free 
group (not necessarily finitely generated). Next, note that 
$A$ is a finitely generated free abelian group 
and hence the FIC$ _{\cal {VC}}$ is true for 
$A$ by Corollary \ref{finabe} and by hypothesis. Now, 
applying Lemma \ref{inverse} to the homomorphism 
$f:\pi_1({\cal G})\to A$ and noting that a torsion 
free virtually cyclic group is either trivial or infinite cyclic 
([\cite{FJ95}, Lemma 2.5])  
we complete the proof. We also need to use Lemma \ref{infin} when $K$ 
is finitely generated.

$1(b)$. The proof of this case is almost the same as that of the previous 
case.

Let $A=H_1(\pi_1({\cal G}), {\Bbb Z})$.
Now, $A$ is a finitely generated abelian group and hence the 
FIC$ _{\cal {VC}}$ is true for $A$ by 
Corollary \ref{finabe}. Next, we apply Lemma \ref{inverse} to 
$p:\pi_1({\cal G})\to A$. Again by Lemma \ref{stark} the kernel $K$ 
of this homomorphism acts on a tree with trivial stabilizers and 
hence $K$ is free. Let $V$ be a virtually cyclic subgroup of 
$A$ with $C<V$ an infinite cyclic subgroup of finite index in 
$V$. Let $C$ be 
generated by $t$. Then the inverse image $p^{-1}(V)$ 
contains $K\rtimes \l t\r$ as a 
normal finite index subgroup. By hypothesis the 
FICwF$ _{\cal {VC}}$ is true for $K\rtimes \l t\r$. 
Now, using $(2)$ of Proposition \ref{product} we see that the 
FICwF$ _{\cal {VC}}$ is true for $p^{-1}(V)$ and 
in particular, the FIC$ _{\cal {VC}}$ is true for 
$p^{-1}(V)$. This completes the proof of $1(b)$.

$1(c)$. Since the graph $\cal G$ is not a tree it is homotopically 
equivalent to a wedge of $r$ circles for $r\geq 1$. Then there is a surjective 
homomorphism 
$p:\pi_1({\cal G})\to F^r$ where $F^r$ is a free group on $r$ generators. 
And the kernel $K$ of this homomorphism is the fundamental group of the 
universal covering $\widetilde {\cal G}$ of the graph of groups $\cal G$. 
Hence, $K$ is the fundamental group of an 
infinite tree of finitely generated abelian groups. Now, we would 
like to apply Lemma \ref{inverse} to the homomorphism $p:\pi_1({\cal 
G})\to F^r$. By hypothesis and by Lemma \ref{infin} the FIC$_{\cal {VC}}$ is true 
for any semidirect product $F^n\rtimes \l t\r$, hence the FIC$_{\cal {VC}}$ is 
true for $F^r$ by the hereditary property. Since $F^r$ is torsion free, 
by Lemma \ref{inverse},   
we only have to check that the FIC$ _{\cal {VC}}$  
is true for the semidirect product $K\rtimes \l t\r$ for any action of 
$\l t\r$ on $K$. 

By Lemma \ref{stark}, the 
restriction of the 
following map to each vertex group of $\widetilde {\cal G}$ is 
injective. $$\pi_1(\widetilde {\cal G})\to H_1(\pi_1(\widetilde 
{\cal G}), {\Bbb Z}).$$ 

Let $A$ and $B$ be the range group and the kernel of the  
homomorphism respectively in the above display.
Since the commutator subgroup of a group is characteristic we 
have the following exact sequence of groups induced 
by the above homomorphism $$1\to 
B\to K\rtimes \l t\r \to A\rtimes \l t\r \to 1$$ for any action of $\l 
t\r$ on 
$K$. Recall that $K=\pi_1(\widetilde {\cal G})$. 
Now, let $K$ acts on a tree $T$, which induces the tree of 
groups structure $\widetilde {\cal G}$ on $K$. Hence, $B$ also acts on $T$ 
with vertex  
stabilizers equal to $B\cap g\widetilde {\cal G}_v g^{-1}=(1)$ 
where $v\in V_{\widetilde {\cal G}}$ and $g\in K$. This follows 
from the fact that the restriction to any vertex group of 
$\widetilde {\cal G}$ of the homomorphism 
$K\to A$ is injective. Thus $B$ acts on 
a tree with trivial stabilizers and hence $B$ is a free group by Lemma 
\ref{finitefree}. 

Next, note that the group $A$ is a countable infinitely generated 
abelian group.
Now, we can apply Lemma \ref{inverse} to the homomorphism $K\rtimes \l t\r 
\to A\rtimes \l t\r$ (and use Lemma \ref{infin} if $B$ is finitely 
generated) and $(2)$ of Proposition \ref{product} in exactly the same 
way as we did in the proof of $1(b)$. This completes the proof of 
$1(c)$.

{\it Proofs of $2(i)$ and $2(ii)$}. 
Let $e_1,e_2,\ldots , e_k$ be the finite edges of $\cal G$ 
so that each of the connected 
components ${\cal G}_1, {\cal G}_2, \ldots ,{\cal G}_n$ of  
${\cal G}-\{e_1,e_2,\ldots , e_k\}$ is a  
tree of finitely generated abelian groups of the same rank. 
By Lemma \ref{ic} a finite 
tree of finitely generated abelian groups of the same rank  
has the intersection property. Therefore, using Lemma \ref{inter} 
we see that such a tree of groups has residually finite fundamental group. 

Now, we check that the FICwF$_{\cal {VC}}$ is true for $\pi_1({\cal G}_i)$ 
for $i=1,2,\ldots , n$. 

Assume that ${\cal G}_i$ is a graph of finitely generated abelian 
groups of the same rank (say $r$). Then by Lemma \ref{ic} $\pi_1({\cal G}_i)$ 
contains a finitely generated free abelian normal subgroup $A$ of rank 
$r$ so that the quotient $\pi_1({\cal G}_i)/A$ is isomorphic to the 
fundamental group of a graph of finite type. Hence,  
by the assumption $_{wt}{\cal T} _{\cal {VC}}$ the 
FICwF$ _{\cal {VC}}$ is true for $\pi_1({\cal G}_i)/A$. 
Now, we would like to apply $(3)$ of Proposition \ref{product} to the 
following exact sequence 
$$1\to A\to \pi_1({\cal G}_i)\to \pi_1({\cal G}_i)/A\to 1.$$ 

Note that using $(2)$ of Proposition \ref{product} it is enough 
to prove that the FICwF$ _{\cal {VC}}$ is true for 
$A\rtimes \langle t\rangle$. In case $2(i)$ this follows from 
the hypothesis and Lemma \ref{infin}. And in case $2(ii)$ note that 
$A\rtimes \langle t\rangle$ contains a rank $2$ free abelian 
subgroup of finite index (Lemma \ref{last}). Therefore, we can again apply 
$(2)$ of Proposition \ref{product} and Lemma \ref{finabe} 
to see that the FICwF$ _{\cal {VC}}$ is true for
$A\rtimes \langle t\rangle$.

Now, observe that there is a finite graph of groups $\cal H$ 
so that the edge groups of $\cal H$ are finite and the 
vertex groups are isomorphic to $\pi_1({\cal G}_i)$ where 
$i$ varies over $1,2,\ldots , n$ and 
$\pi_1({\cal H})\simeq \pi_1({\cal G})$. This follows 
from Lemma \ref{almost} since $\cal G$ is almost a tree of groups.

Next, we apply 
$(1)$ of Proposition \ref{residually-prop} 
(since $\pi_1({\cal G}_i)$ is residually finite for $i=1,2,\ldots , n$)  
to $\cal H$ to complete the proof of $(2)$.  

{\it Proof of $(3)$}. First, assume that the tree is finite. By Lemma 
\ref{stark}, there is 
a homomorphism $\pi_1({\cal G})\to H_1(\pi_1({\cal G}), {\Bbb Z})$ whose restriction to any  
vertex group is injective. A limit argument proves the 
FICwF$_{\cal {VC}}$ for $H_1(\pi_1({\cal G}), {\Bbb Z})$ in case some vertex group is infinitely 
generated. Now, $(2)$ can be applied to complete the proof of 
$(3)$.\end{proof}

\begin{proof} [Proof of Proposition \ref{TC}] Lemma \ref{almost} implies that 
there is a graph of groups $\cal H$ with the same fundamental group as 
$\cal G$ so that the edge groups of $\cal H$ are finite and the vertex groups are either 
finite or fundamental groups of finite trees of groups with rank $1$ finitely 
generated abelian vertex and edge groups. Now, note that by Lemma 
\ref{ic} the tree of group (say $\cal K$) corresponding to 
a vertex group of $\cal H$ satisfies the hypothesis of Lemma \ref{fi}. 

Therefore, $\cal K$ has the intersection property. Also by Lemma 
\ref{rfcyclic} $\pi_1(\cal K)$ is residually finite. Since by 
Corollary \ref{PVC} the FICwF$^{P}$ 
is true for virtually cyclic groups, we can
apply $(1)$ and $(3)$ of Proposition \ref{residually-prop} to complete
the proof of the Proposition provided we check that $_{wt}{\cal
T}^{P}$ and ${\cal
P}^{P}$ are satisfied. By Corollary 
\ref{PVC} the last condition is satisfied. 

We now check the first 
condition. So let $\cal G$ be a graph of virtually cyclic 
groups with trivial edge groups. Using Remark \ref{finitegraph} 
we can assume that the graph is finite. Hence, by Lemma \ref{almost1} below,  
$\pi_1({\cal G})$ is a free product of a finitely generated 
free group and the vertex 
groups of $\cal G$. By Corollary \ref{PVC}, 
the FICwF$^{P}$ is true for any virtually 
cyclic group and by Proposition \ref{graph-finite} and 
Lemma \ref{finitefree} it is true for finitely generated free groups. 
Therefore, we can apply [\cite{R1}, Theorem 3.1] to conclude that the 
FICwF$^{P}$ is true for $\pi_1({\cal G})$.\end{proof}

We remark that, when we applied $(3)$ of Proposition \ref{residually-prop} 
in the 
above proof of Proposition \ref{TC}, we only needed the fact that 
the FICwF$^P$ be true for the mapping torus of a virtually 
cyclic group (see the proof of $(3)$ of Proposition \ref{residually-prop}). 
We note that for this we do not need Proposition \ref{polycyclic}, instead 
it can easily be deduced from the proof of Corollary \ref{PVC}. We just need 
to mention that the mapping torus of a virtually cyclic group is either 
virtually cyclic or it contains a two-generator 
free abelian normal subgroup of finite index (see the following lemma). 

\begin{lemma} \label{last} Suppose a group $G$ has a virtually cyclic normal subgroup 
$V_1$ with virtually cyclic quotient $V_2$. Then $G$ is virtually cyclic, when 
either $V_1$ or $V_2$ is finite, and otherwise $G$ contains a two-generator free abelian normal 
subgroup of finite index.\end{lemma}

\begin{proof} We have the following exact sequence $$1\to V_1\to G\to V_2\to 1.$$ 
For the first assertion there is nothing to prove when $V_2$ is finite.
So let $V_1$ be finite and $V_2$ infinite and let $C$ be an infinite cyclic subgroup of 
$V_2$ of finite index. Let $p$ be the 
homomorphism $G\to V_2$. Then $p^{-1}(C)$ 
has a finite normal subgroup with quotient $C$. Therefore, $p^{-1}(C)$ 
is virtually infinite cyclic. Let $C'$ be an infinite cyclic 
subgroup of $p^{-1}(C)$ 
of finite index. Hence, $C'$ is also an infinite cyclic 
subgroup of $G$ of finite 
index. This proves the first assertion. For the second assertion let $V_1$ be also infinite 
and $D$ an infinite cyclic subgroup of $V_1$ of finite index. Then 
$p^{-1}(C)\simeq D\rtimes C$. Since $C$ and $D$ are both infinite cyclic the 
only possibilities for $p^{-1}(C)$ are that it contains a free abelian subgroup on two 
generators of index either $1$ or $2$. This proves the Lemma.\end{proof}

\begin{proof} [Proof of Proposition \ref{theorem}] Let $T$ be the tree on
which the group $\pi_1({\cal G})$ acts so that the associated graph of
groups is $\cal G$. For the proof, by 
Lemma \ref{inverse} ($(3)$ of Proposition 
\ref{product}) we need to show 
that the FIC$ _{\cal C}$ (FICwF$ _{\cal C}$) 
is true for the homomorphism 
$f$. We prove $(a)$ assuming ${\cal
T} _{\cal C}$ ($_w{\cal
T} _{\cal C}$). 
 
Let $H\in {\cal C}(Q)$. Note that $f^{-1}(H)$ also acts on the tree $T$ with 
stabilizers $f^{-1}(H)\cap\{$stabilizers of the action of $\pi_1({\cal G})\}$. 
Since the restrictions of $f$ to the vertex groups  
of ${\cal G}$ have finite kernels we get that 
$f^{-1}(H)\cap\{$stabilizers 
of the action of $\pi_1({\cal G})\}$ is an extension of a finite group 
by a subgroup of $H$. If ${\cal C}={\cal {FIN}}$ then these stabilizers 
also belong to $\cal C$. If ${\cal C}={\cal {VC}}$ then 
using Lemma \ref{last} we see that the stabilizers again belong to 
$\cal C$. 

Therefore, the associated graph of groups of the action of $f^{-1}(H)$ 
on the tree $T$ has vertex groups belonging to $\cal C$. 

Now, using ${\cal T} _{\cal C}$ ($_w{\cal
T} _{\cal C}$) we 
conclude that 
the FIC$_{\cal C}$ (FICwF$_{\cal C}$) is true for $f^{-1}(H)$. This completes the 
proof of $(a)$.

For the proofs of $(b)$ and $(c)$ just replace  
${\cal T} _{\cal C}$ ($_w{\cal
T} _{\cal C}$) by  
$_f{\cal T} _{\cal C}$ ($_{wf}{\cal
T} _{\cal C}$) and  
$_t{\cal T} _{\cal C}$ ($_{wt}{\cal
T} _{\cal C}$) respectively in the above proof. 
Also use the corresponding assumption on the edge groups of the graph 
of groups $\cal G$.

\end{proof}

\begin{cor} \label{corollary} {\bf (Free products).}
Let $\cal G$ be a finite graph of groups with trivial 
edge groups so that the vertex groups satisfy the FIC$_{\cal C}$ 
(FICwF$_{\cal C}$). If ${\cal P} _{\cal 
C}$ and $_t{\cal T} _{\cal 
C}$ ($_{wt}{\cal T} _{\cal
C}$) are  
satisfied then the FIC$_{\cal C}$ (FICwF$_{\cal C}$) 
is true for $\pi_1({\cal G})$.\end{cor} 

\begin{proof} The proof combines the following two Lemmas.\end{proof}

\begin{lemma} \label{almost1} Let $\cal G$ be a finite graph 
of groups with trivial edge groups. Then there is an
isomorphism
$\pi_1({\cal G})\simeq G_1*\cdots *G_n*F$
where $G_i$'s are vertex groups of ${\cal G}$ 
and $F$ is a free group.\end{lemma}

\begin{proof} The proof is by
induction on the number of edges of the graph. If the graph
has no edge then there is nothing to prove. So assume
${\cal G}$ has
$n$ edges and that the lemma is true for graphs with $\leq
n-1$ edges.
Let
$e$ be an edge of ${\cal G}$. If ${\cal G}-\{e\}$ is connected
then $\pi_1({\cal G})\simeq \pi_1({\cal
G}_1)*{\Bbb Z}$ where ${\cal G}_1={\cal G}-\{e\}$ is a graph
with $n-1$
edges. On the other hand if ${\cal G}-\{e\}$ has two
components say ${\cal
G}_1$ and ${\cal G}_2$ then $\pi_1({\cal G})\simeq \pi_1({\cal
G}_1)* \pi_1({\cal G}_2)$ where ${\cal G}_1$ and ${\cal G}_2$
has $\leq
n-1$ edges. Therefore, by induction we complete the proof of the
lemma.\end{proof}

\begin{lemma} \label{claim} Assume that the properties ${\cal P} _{\cal
C}$ and $_t{\cal T} _{\cal
C}$ ($_{wt}{\cal T} _{\cal
C}$) are
satisfied. If the FIC$ _{\cal C}$ 
(FICwF$ _{\cal C}$) is true 
for  
$G_1$ and $G_2$ then the FIC$ _{\cal C}$ 
(FICwF$ _{\cal C}$) 
is 
true for $G_1 * G_2$.\end{lemma}

\begin{proof} Consider the surjective homomorphism $p:G_1 * G_2\to G_1 
\times G_2$. By $(1)$ of Proposition \ref{product} the 
FIC$ _{\cal C}$ (FICwF$ _{\cal C}$) 
is true for $G_1\times G_2$. 

Now, note that the group $G_1 * G_2$ acts on a tree with 
trivial edge stabilizers and vertex stabilizers conjugates of $G_1$ or 
$G_2$. Therefore, the restrictions of $p$ to the stabilizers of this action 
of $G_1*G_2$ on the tree are injective. Hence, we are in the situation of 
 $(c)$ of Proposition \ref{theorem},

This completes the proof of the Lemma.\end{proof} 

\begin{rem}{\rm We recall that in the case of the 
FICwF$^{P}$, Lemma \ref{claim} coincides with the Reduction Theorem 
(Theorem 3.1) in 
\cite{R1}.}\end{rem}

\begin{proof}[Proof of Proposition \ref{residually-prop}] Since the 
equivariant homology theory is assumed to be continuous when the graph 
is infinite, we can assume
that $\cal G$ is a finite graph of groups.

\noindent
$(1).$ By hypothesis the edge groups of $\cal G$ are finite and the
vertex groups 
are residually finite and satisfy the FICwF$ _{\cal 
{C}}$. By Lemma \ref{graph-res} $\pi_1({\cal G})$ is residually finite. Let 
$F_1,F_2,\ldots ,F_n$ be the edge groups. 
Let $1\neq g\in\cup_{i=1}^nF_i$. Since 
$\pi_1({\cal G})$ is residually finite there is a finite-index normal subgroup 
$N_g$ of 
$\pi_1({\cal G})$ so that $g\in \pi_1({\cal G})-N_g$. Let $N=\cap_{g\in 
\cup_{i=1}^nF_i}N_g$. Then $N$ is a finite-index normal subgroup of
$\pi_1({\cal G})$ so that $N\cap (\cup_{i=1}^nF_i)=\{1\}$. 

Let $T$ be a tree on which $\pi_1({\cal G})$ acts so that the 
associated graph of group structure on $\pi_1({\cal G})$ is $\cal G$. Hence, $N$ 
also acts on $T$. Since $N$ is normal in $\pi_1({\cal G})$ and 
$N\cap (\cup_{i=1}^nF_i)=\{1\}$, the edge stabilizers of 
this action are trivial and the vertex stabilizers are subgroups of
conjugates of 
vertex groups of $\cal G$. Therefore, by Corollary \ref{corollary} the FICwF$_{\cal 
{C}}$ is true for $N$. Now, $(2)$ of Proposition \ref{product} completes
the proof 
of $(1)$.

\noindent
$(2).$ The proof is by induction on the number of edges. If 
there is no edge then there is one vertex and hence by 
hypothesis the induction starts. Since by hypothesis  
the FICwF$_{\cal {VC}}$ is true for $C\rtimes \l t\r$ for 
$C\in {\cal D}$ it is true for $C$ also 
by Lemma \ref{finiteindex}.
Therefore, assume that the result is true for graphs with $\leq n-1$ edges, 
which satisfy the hypothesis.
So let $\cal G$ be a finite graph of groups which satisfies the
hypothesis and has $n$ edges. 
Since $\cal G$ has the intersection property there is a normal subgroup
$N$ of $\pi_1({\cal G})$ contained in all the edge groups and of finite
index in some edge group, say ${\cal G}_e$ of the edge $e$. 
Let ${\cal G}_1$ be the 
graph of groups with $\cal G$ as the underlying graph and the vertex
and the edge groups are ${\cal G}_x/N$ where $x$ is a vertex and an edge  
respectively. Then $\pi_1({\cal G}_1)\simeq \pi_1({\cal G})/N$. 
Let ${\cal G}_2={\cal G}_1-\{e\}$. It is now easy to check that 
the connected components of ${\cal G}_2$ satisfy the
hypotheses and also have the intersection property.
Also, since ${\cal G}_2$ has $n-1$ edges, by the induction hypothesis 
$\pi_1({\cal H})$ satisfies the FICwF$_{\cal 
{VC}}$ where $\cal H$ is a connected component of ${\cal G}_2$. 
We now use Lemma \ref{inter} to conclude that $\pi_1({\cal H})$
is also residually finite where $\cal H$ is as above. 
Therefore, by $(1)$ $\pi_1({\cal G}_1)$
satisfies the FICwF$ _{\cal 
{VC}}$. Next, we apply $(3)$ of Proposition \ref{product} to the homomorphism 
$\pi_1({\cal G})\to \pi_1({\cal G}_1)$.  Note that using $(2)$ of
Proposition \ref{product} it is enough to show the FICwF$_{\cal
{VC}}$ is true for the inverse image of any infinite cyclic 
subgroup of $\pi_1({\cal G}_1)$, but such a group is of the 
form $N\rtimes \l t\r$. Now, since $N$ is a subgroup of the edge 
groups of $\cal G$ and $\cal D$ is closed under 
taking subgroups we get $N\in \cal D$, as by hypothesis all the 
vertex groups of $\cal G$ belong to $\cal D$. Again by definition 
of $\cal D$ the FICwF$_{\cal
{VC}}$ is true for $N\rtimes \l t\r$. This 
completes the proof of $(2)$.
 
\noindent
$(3).$ The proof of $(3)$ follows from $(2)$. For the polycyclic case we
only need to note that virtually polycyclic groups are residually finite
and quotients and subgroups of virtually polycyclic groups are 
virtually polycyclic. Also the   
mapping torus of a virtually polycyclic group is again 
virtually polycyclic.

For the nilpotent case, recall from \cite{MR} that the
fundamental group of a finite graph of finitely generated nilpotent
groups is subgroup separable if and only if the graph of groups satisfies
the intersection property. Also note that finitely generated nilpotent
groups are virtually polycyclic.

\noindent
$(4).$ Note that by hypothesis the FICwF$ _{\cal {VC}}$ 
is true for closed surface groups and by 
\cite{GB1} closed surface groups 
are residually finite. Therefore, by Lemma \ref{res-surface} 
and by $(1)$ it is enough
to consider a finite graph of groups whose vertex and edge groups are
infinite closed surface groups and the graph of groups satisfies the hypothesis. 
Again by Lemma \ref{res-surface} we have
the exact sequence: $1\to H\to \pi_1({\cal G})\to \pi_1({\cal
G})/H\to 1$ where $H$ is a closed surface group and $\pi_1({\cal  
G})/H$ is isomorphic to the fundamental group of a finite-type graph of 
groups and hence contains a finitely generated free subgroup of finite
index by Lemma \ref{finitefree}. Hence, by $_{wt}{\cal T}_{\cal {VC}}$ and by $(2)$ of Proposition \ref{product}
the FICwF$ _{\cal {VC}}$ is true for $\pi_1({\cal
G})/H$. Now, apply $(3)$ of Proposition \ref{product} to the homomorphism
$\pi_1({\cal G})\to \pi_1({\cal  
G})/H$. Using $(2)$ of Proposition \ref{product} it is enough to prove the
FICwF$ _{\cal {VC}}$ for $H\rtimes \l t\r$. Since $H$ is a
closed surface group the action of $\l t\r$ on $H$ is induced by a
diffeomorphism of a surface $S$ so that $\pi_1(S)\simeq H$. Hence,
$H\rtimes \l t\r$ is isomorphic to the fundamental group of a closed
$3$-manifold which fibers over the circle. Therefore, using the hypothesis 
we complete the proof.\end{proof}

\subsection{Proofs of the theorems}

\begin{proof}[Proof of Theorem \ref{residually}] By Corollary \ref{tt}
 $_{wf}{\cal T}^{P}$ is satisfied and hence $_{wt}{\cal T}^{P}$ 
is also satisfied. Next, by Corollary, \ref{PVC} ${\cal
P}^{P}$ is satisfied. Therefore, the
proof of $(1)$ is completed using $(1)$ of Proposition 
\ref{residually-prop}.

For the proof of $(2)$ apply $(b)$ of Proposition \ref{theorem} 
and Corollary \ref{tt}.
\end{proof}

\begin{rem}{\rm Using Proposition \ref{polycyclic},  
the same proof works to prove the more general statement in case $(2)$, when the 
kernel of $f$ restricted to any vertex group is virtually polycyclic.}\end{rem}

\begin{proof}[Proof of Theorem \ref{ip}] Corollaries \ref{tt}, \ref{PVC}
and $(2)$ of Proposition \ref{residually-prop} prove the theorem.
\end{proof}

\begin{proof}[Proof of Theorem \ref{introthm}] Corollaries \ref{tt},
\ref{PVC}, Proposition \ref{polycyclic} and $(3)$ of 
Prop-\\osition \ref{residually-prop} prove $(1)$ and
$(2)$ of the theorem. 

To prove $(3)$ we need only to use $(4)$ of
Proposition \ref{residually-prop}, Corollary \ref{tt} and that the
FICwF$^{P}$ is true for the
fundamental groups of $3$-manifolds fibering over the circle from 
\cite{R1} and \cite{R2}.

The proof of $(4)$ follows from $2(i)$ of 
Proposition \ref{maintheorem1}, 
Corollaries \ref{tt} and \ref{PVC}. Since the FICwF$ _{\cal {VC}}$ 
is true for ${\Bbb Z}^n\rtimes \langle t\rangle$ for all $n$ 
by Proposition \ref{polycyclic}.

For the proof of $(5)$ first assume that the graph is finite and has more than one edge. Then get a 
homomorphism $f:\pi_1({\cal G})\to A:=H_1(\pi_1({\cal G}), {\Bbb Z})$ using Lemma 
\ref{stark} where the restriction of $f$ to any vertex group is injective. 
Now, we apply Corollaries \ref{tt} and \ref{PVC} and Proposition 
\ref{product} to the above homomorphism. Therefore, one only needs that $f^{-1}(C)$ 
for any infinite cyclic subgroup $C$ of $A$ satisfies the FICwF$^P$. But 
the restricted action of $f^{-1}(C)$ on the tree (on which $\pi_1({\cal G})$ acts) 
has infinite cyclic or trivial vertex stabilizers. Therefore, we conclude the proof using 
the Reduction Theorem in \cite{R1} and the hypothesis.

Next, if the graph has one edge then $\pi_1({\cal G})\simeq G_1*_HG_2$ where $G_1, G_2$ and 
$H$ are finitely generated abelian. Then $H$ is a normal subgroup of $\pi_1({\cal G})$ 
with quotient $G_1/H*G_2/H$. The proof now follows the similar steps as in the previous 
case using Proposition \ref{product}.
\end{proof}

\begin{proof}[Proof of Proposition \ref{k-l-theory}] Since the conjecture in $L$-theory 
and lower $K$-theory commutes with direct limits (Remark \ref{finitegraph}) 
we can assume that the tree of 
virtually cyclic groups is finite. Also since the conjecture is true for 
virtually cyclic groups we can start proving the proposition by induction 
on the number of edges. We need the following lemma.

\begin{lemma} Let $G*_VG'$ be an amalgamated free product where $G$ and $G'$ act  
on finite-dimensional $CAT(0)$ spaces properly, cocompactly and isometrically and $V$ is a  
virtually cyclic group. Then $G*_VG'$ also acts on a finite-dimensional $CAT(0)$
space properly, cocompactly and isometrically.\end{lemma}

\begin{proof} The lemma follows easily from the proof of 
[\cite{BH}, Corollary 11.19, p. 357].\end{proof}

The proof of the proposition now follows from \cite{BL1} and noting that any 
virtually cyclic group acts on a finite-dimensional $CAT(0)$ space properly, 
cocompactly and isometrically.\end{proof}

\subsection{Some deductions}
We now show that a positive solution to the Problem in the 
beginning of this section will imply that the conjecture is true for solvable and for 
one-relator groups. We will have to show that both classes of groups 
can be obtained, in some suitable sense, from groups which act on trees so 
that the vertex stabilizers satisfy the conjecture.

\noindent
{\bf Solvable groups.} Recall that in \cite{R5} we had shown that the FICwF 
in both pseudo-isotopy and $L$-theory is true for a virtually solvable group 
if and only if the same conjecture is true for closely crystallographic groups. 
A closely crystallographic group is by definition of the form 
$A\rtimes C$ where $A$ is a torsion free abelian group and $C$ is an 
infinite cyclic group so that $A$ is an irreducible ${\Bbb Q}[C]$-module. 
Now, note that the semidirect product structure gives an action of the 
closely crystallographic group on a tree with stabilizers isomorphic to $A$. 
Since the conjecture is true for $A$, a positive solution to the 
Problem will imply the conjecture for any virtually solvable group.

\noindent
{\bf One-relator groups.} Let $G$ be a one-relator group. Then 
$G\simeq H*F$ where $H$ is a finitely generated one-relator group and 
$F$ is a free group. Therefore, by [\cite{R5}, Reduction Theorem], 
Corollary \ref{tt} and Proposition \ref{integer}, we can assume that $G$ 
is finitely generated. Let $l$ be the length of the relator in the group $G$. 
Then by a result of Bieri (\cite{Bi}) $G$  
acts on a tree with stabilizers subgroups of one-relator groups with the relator 
lengths $\leq l-1$. Since a one-relator group with relator length 
$1$ is a free product of a free group and a finite cyclic group, by induction 
on $l$, using again [\cite{R5}, Reduction Theorem], Corollary \ref{tt} and 
Proposition \ref{integer} and using a positive solution to the problem, 
we deduce the conjecture 
for all one-relator groups. 

\subsection{Examples} 

Below we describe some examples of groups 
for which the results in this article are applicable. Furthermore, 
we show that these groups are new and that they are neither 
$CAT(0)$ nor hyperbolic.

\begin{exm}\label{end1}{\rm {\bf (Almost) a tree of  
finitely generated
abelian groups where the vertex and the edge groups of any 
component subgraph have the same rank:} Fundamental group of 
such a graph of groups can get very
complicated, for example in the simplest case of amalgamated
free product of two infinite cyclic groups over an infinite
cyclic group, that is $H=Z*_Z Z$ where $Z$ is infinite cyclic, 
produces the $(p,q)$-torus knot group where the two inclusions
$Z\to Z$ defining the amalgamation are power raised by $p$ and
$q$, $(p,q)=1$. Though the FICwF$^P$ is
known for knot groups (\cite{R1}), most of the other groups in this class 
for which we prove the FICwF$^P$ are new. We further recall that 
in general these examples of groups where the edge groups have  
rank $\geq 2$ are not $CAT(0)$ (see the discussion after the 
proof of Proposition 6.8 on page 500 in \cite{BH}).}\end{exm}

\begin{exm}\label{end}{\rm 
{\bf Graphs of virtually polycyclic groups:} We consider an 
amalgamated free product $H$ of two nontrivial infinite virtually polycyclic groups
over a finite group. Next, we recall that in [6] the fibered 
isomorphism conjecture in the pseudo-isotopy case was
proved for the class of groups 
which act cocompactly and properly discontinuously on symmetric 
simply connected nonpositively curved
Riemannian manifolds. In the proof of  
[\cite{BFJP}, Theorem A] it was noted that the condition `symmetric' 
can be replaced by `complete' if we consider torsion free groups. 
See also [\cite{FR}, Theorem A].
We choose polycyclic groups
so that $H$ does not belong to this class. One example
of such a polycyclic group $S$ is of the type $1\to Z\to S\to  Z^2\to 1$ as
described below, where $Z$ is infinite cyclic.

Let G be the Lie group consisting of those
$3 \times 3$ matrices with real entries whose diagonal entries are all
equal to $1$, entries below the diagonal are all equal to $0$,
and entries above the diagonal are arbitrary. Note that G is
diffeomorphic
to Euclidean 3-space. Let S be the subgroup of G whose entries above the
diagonal are restricted to be integers. Then S is discrete and cocompact.
Let E be the coset space of G by S. It clearly fibers over the 2-torus
with fiber the circle. And the fundamental group of E is S, which is
nilpotent
but not abelian. On the other hand, in \cite{Y} it was 
shown that the fundamental group of a closed nonpositively curved
manifold, which is nilpotent must be abelian. This shows that
$E$ cannot support a nonpositively curved Riemannian metric. Now, by 
[\cite{KL}, Corollary 2.6] it follows that $S$ cannot even  
embed in a  group (called {\it Hadamard group}), which acts  
discretely and cocompactly on a 
complete simply connected nonpositively curved space (that is a $CAT(0)$-space).
 
Now, consider $H_i=S\t F_i$ or $H_i=S\wr F_i$ where $F_i$ 
is a finite group for $i=1,2$.
Next, we take the amalgamated free product of $H_1$ and $H_2$, $H=H_1*_FH_2$ along some
finite group $F$. Then $H$ does not embed in a Hadamard group as before by 
[\cite{KL}, Corollary 2.6] and $H$ is not virtually 
polycyclic. 
But $H$ satisfies the FICwF$^P$ by $(1)$ of Theorem 1.1 and Proposition 5.4.}\end{exm}

\begin{exm}\label{end0}{\rm {\bf Graphs of residually finite groups with 
finite edge grou-\\ps:} Let $S$ be the fundamental group of a compact 
Haken $3$-manifold, which does not support any nonpositively curved 
Riemannian metric. Such $3$-manifolds can easily be constructed by 
cutting along an incompressible torus in a compact Haken 
$3$-manifold and then gluing differently. See \cite{KL} 
for this kind of construction. Next,  
let $H_1$ and $H_2$ be two residually finite 
groups for which the FICwF$^P$ is true and such that $S$ is embedded in 
$H_1$. It is easy to construct such $H_1$, for 
instance take $H_1=S*G*F_1$ or  
$H_1=(S\t G)\wr F_1$ or any such combination where $G$ 
is a finitely generated free group and $F_1$ is a finite group. 
By the same argument as in Example \ref{end} (and 
using \cite{R1} and \cite{R2}) it 
follows that $H=H_1*_FH_2$ (along some finite group $F$) 
satisfies the FICwF$^P$ but is neither virtually polycyclic nor embeds 
in a Hadamard group.}\end{exm}

Let us now note that the group $H$ considered in the above examples 
is also not hyperbolic as it contains a free abelian subgroup on more 
than one generator. 

Finally, we remark that in a recent paper (\cite{BL1}) the fibered 
isomorphism conjectures in $L$- and lower $K$-theory are proved  
for hyperbolic groups and for $CAT(0)$-groups, which act  
on finite-dimensional $CAT(0)$-spaces.

\subsection{Open problems}
We state some open problems related to this article on the 
(fibered) isomorphism conjecture.

\noindent
{\bf A.} Show that the (fibered) isomorphism conjecture is true for 
$A\rtimes {\Bbb Z}$ for a 
torsion-free abelian group $A$ and for an arbitrary action of ${\Bbb Z}$ on $A$.
Note that a positive answer to this problem will imply the conjecture for all 
solvable groups. See \cite{R5}.

\noindent
{\bf B.} This is a more general situation compared to {\bf A}. Show that the 
conjecture is true for $G\rtimes {\Bbb Z}$ assuming the conjecture for $G$. 
This is a very important open problem and will imply the conjecture for 
poly-free groups. It is open even when $G$ is finitely generated and free. 
For certain situations 
the answers are known, for example, when $G$ is a surface group and the action 
is realizable by a diffeomorphism of the surface. See \cite{R5}.

\noindent
{\bf C.} Prove $(1)$ of Theorem \ref{residually} without the assumption `residually finite'. 
This will imply that one only needs to prove the fibered isomorphism conjecture 
for finitely presented groups with one end. 

\noindent
{\bf D.} Prove the (fibered) isomorphism conjecture for the fundamental group of a 
graph of virtually cyclic groups. Even for the graph of infinite cyclic groups this is 
an open problem. If the underlying graph is a tree, then in Proposition \ref{k-l-theory} 
we proved it for the $L$-theory and lower $K$-theory case.

\begin{rem}\label{lastremark}{\rm Finally, we remark  
that in this section, [\cite{FJ}, Theorem 4.8] is used 
(see Proposition \ref{polycyclic}) in the proofs of 
$(1)$, $(2)$ (when the polycyclic or the nilpotent groups are 
not virtually cyclic), $(3)$ and $(5)$ (when the ranks of the abelian groups are $\geq 2$) 
of Theorem \ref{introthm}. See the proof of Corollary \ref{PVC} and the 
discussion after the proof of Proposition \ref{TC}. In this connection we note 
here that using the recent work   
in \cite{BL1} all the results in this section can 
be deduced in the $L$-theory case of the fibered isomorphism  conjecture. 
The same proofs will go through.  
But for this we need to use the $L$-theory version of [\cite{FJ}, Theorem 4.8]  
in the proofs of the particular cases of the items  
of Theorem \ref{introthm} as mentioned above. See \cite{BFL} for 
the proof of [\cite{FJ}, Theorem 4.8] in the $L$-theory case.}\end{rem}

\medskip
\noindent
{\bf Acknowledgment} 
\medskip

Part of the work in Section 3 was supported by a 
fellowship of the Alexander von Humboldt Foundation, Germany.

%\newpage
\bibliographystyle{plain}
\ifx\undefined\bysame
\newcommand{\bysame}{\leavevmode\hbox to3em{\hrulefill}\,}
\fi

\end{document}